\documentclass[reqno]{amsart}

\usepackage{amssymb}
\usepackage{amsmath}
\usepackage{amsfonts}
\usepackage[colorlinks]{hyperref}
\usepackage{mathtools}
\usepackage{graphicx}
\usepackage{mathptmx}
\usepackage{bm}
\usepackage{physics}
\usepackage{csquotes}
\usepackage[thinc]{esdiff}
\usepackage{enumerate}

\numberwithin{equation}{section}

\theoremstyle{definition}
\newtheorem{definition}{Definition}

\newtheorem{myexample}{Example}
\newenvironment{example}{\pushQED{\qed}\myexample}{\popQED\endmyexample}
\theoremstyle{plain}
\newtheorem{lemma}{Lemma}
\newtheorem{theorem}{Theorem}
\theoremstyle{remark}
\newtheorem{remark}{Remark}

\allowdisplaybreaks

\newcommand{\aast}{a^{\star}}
\newcommand{\assumptionRef}[1]{\textrm{Assumption (\ref{#1})}}
\newcommand{\assumRef}[1]{\textrm{Assum.\,}\br{\textrm{\ref{#1}}}}
\newcommand{\axiomRef}[1]{\textrm{Axiom (\ref{#1})}}
\newcommand{\axRef}[1]{\textrm{Ax.\,(\ref{#1})}}
\newcommand{\ball}[3]{K_{#1}\sbr{#2,#3}}
\newcommand{\BC}[2]{BC\br{#1,#2}}
\newcommand{\bff}{\mathbf{f}}
\newcommand{\bp}{\mathbf{p}}
\newcommand{\br}[1]{\left(#1\right)}
\newcommand{\brt}[1]{\br{#1}_{\T}}
\newcommand{\bast}{b^{\star}}
\newcommand{\bu}{\mathbf{u}}
\newcommand{\bvarphi}{\boldsymbol{\varphi}}
\newcommand{\bx}{\mathbf{x}}
\newcommand{\by}{\mathbf{y}}
\newcommand{\C}[2]{C\br{#1,#2}}
\newcommand{\cbr}[1]{\left\{#1\right\}}
\newcommand{\cl}[1]{\clOperator\br{#1}}
\newcommand{\conv}[1]{\convOperator\br{#1}}
\newcommand{\Crd}[2]{C_{rd}\br{#1,#2}}
\newcommand{\CrdOne}[2]{C_{rd}^1\br{#1,#2}}
\newcommand{\defeq}{\vcentcolon=}
\newcommand{\definitionRef}[1]{\textrm{Definition \ref{#1}}}
\newcommand{\defRef}[1]{\textrm{Def.\,\ref{#1}}}
\newcommand{\DeltaAint}[3]{\br{A}\!\Deltaint{#1}{#2}{#3}}
\newcommand{\Deltaint}[3]{\int_{#1}^{#2}{#3}}
\newcommand{\Deltas}{\mathrm{\Delta}s}
\newcommand{\Deltatau}{\mathrm{\Delta}\tau}
\newcommand{\E}{E}
\newcommand{\Est}{\E^{\star}}
\newcommand{\eps}{\varepsilon}
\newcommand{\eqdef}{=\vcentcolon}
\newcommand{\eqText}[1]{\stackrel{#1}{=}}
\newcommand{\exampleRef}[1]{\textrm{Example \ref{#1}}}
\newcommand{\fF}{\mathcal{F}}
\newcommand{\fM}{\mathcal{M}}
\newcommand{\fN}{\mathcal{N}}

\newcommand{\geqText}[1]{\stackrel{#1}{\geq}}
\newcommand{\lbrsbr}[1]{\left(#1\right]}
\newcommand{\lbrsbrt}[1]{\lbrsbr{#1}_{\T}}
\newcommand{\lemmaRef}[1]{\textrm{Lemma \ref{#1}}}
\newcommand{\leqText}[1]{\stackrel{#1}{\leq}}
\newcommand{\lmRef}[1]{\textrm{Lm.\,\ref{#1}}}
\newcommand{\lsrbr}[1]{\left[#1\right)}
\newcommand{\lsrbrt}[1]{\lsrbr{#1}_{\T}}
\newcommand{\lText}[1]{\stackrel{#1}{<}}
\newcommand{\N}{\mathbb{N}}
\newcommand{\nbrt}[2]{\mathcal{O}_{\T}^{#1}\br{#2}}
\newcommand{\openball}[3]{K_{#1}\br{#2,#3}}
\newcommand{\partRef}[1]{\textrm{Part \ref{#1}}}
\newcommand{\pDelta}[2]{\frac{\partial{#1}}{\Delta{#2}}}
\newcommand{\R}{\mathbb{R}}
\newcommand{\remarkRef}[1]{\textrm{Remark \,\ref{#1}}}
\newcommand{\sbr}[1]{\left[#1\right]}
\newcommand{\sbrt}[1]{\sbr{#1}_{\T}}
\newcommand{\sbrtk}[1]{\sbrt{#1}^{\kappa}}
\newcommand{\sectionRef}[1]{\textrm{Section \ref{#1}}}
\newcommand{\st}{\!:}
\newcommand{\subseteqText}[1]{\stackrel{#1}{\subseteq}}
\newcommand{\T}{\mathbb{T}}

\newcommand{\theoremRef}[1]{\textrm{Theorem \ref{#1}}}
\newcommand{\Tk}{\T^{\kappa}}
\newcommand{\thmRef}[1]{\textrm{Thm.\,\ref{#1}}}
\newcommand{\Z}{\mathbb{Z}}

\DeclareMathOperator{\clOperator}{cl}
\DeclareMathOperator{\convOperator}{conv}

\begin{document}
	\title[On the existence of solutions of dynamic equations on time scales in Banach spaces]{On the existence of solutions of dynamic equations on time scales in Banach spaces}
	\author[D. Oberta]{Dušan Oberta}
	\address{Institute of Mathematics, Faculty of Mechanical Engineering, Brno University of Technology, Technická 2, 616 69 Brno, Czech Republic}
	\email{Dusan.Oberta@vutbr.cz, oberta.du@gmail.com}
	\begin{abstract}
		In this paper we address the question of solvability of dynamic equations on time scales in Banach spaces. In particular, our main theorem extends the result for classical differential equations in Banach spaces of Banaś and Goebel (1980) established in \cite{Banas_Goebel}, to an arbitrary time scale. Central role is played by the axiomatic theory of measures of noncompactness and the newly introduced Kamke $\Delta$-function. Also, we study countable systems of dynamic equations on time scales arising from semi-discretisation of parabolic partial dynamic equations.
	\end{abstract}
	\maketitle

\section{Introduction}
Calculus on time scales, as introduced in \cite{Hilger}, provides not only a unification between the theories of classical differential and difference equations, but also their extension to more general settings. Moreover, as the original theory was developed in the context of Banach spaces, it is natural to study dynamic equations on time scales in Banach spaces directly.

Let $\E$ be a real Banach space and $\T$ a time scale. In this paper we investigate the following initial value problem
\begin{subequations}\label{eq_01_01}
	\begin{align}
		u^{\Delta}\br{t}&=f\br{t,u\br{t}},\quad{t}\in\sbrtk{a,b},\\
		u\br{a}&=u_0,
	\end{align}
\end{subequations}
where $u_0\in\E$, $f\st\sbrt{a,b}\times\E\to\E$, and $u^{\Delta}\br{\cdot}$ denotes the $\Delta$-derivative of $u$.

The basic question of existence and uniqueness of local solutions of \eqref{eq_01_01} is well-established, as the Picard-like theorem (i.e. the existence and uniqueness theorem with $f$ being \enquote{some sort of rd-continuous}, and Lipschitz continuous in the second variable) holds for an arbitrary time scale and an arbitrary Banach space (see \cite{Bohner_Dynamic}).

On the other hand, the basic question of existence of local solutions of \eqref{eq_01_01} is much more problematic. It is well-known that the Peano-like theorem (i.e. the existence theorem with $f$ being \enquote{some sort of rd-continuous}) does not necessarily hold. In particular, it is a well-known result that it fails for classical differential equations in infinite-dimensional Banach spaces, whereas it holds for difference equations in an arbitrary Banach space. Furthermore, as pointed out in \cite{Cichon}, such a Peano-like theorem holds for an arbitrary time scale and an arbitrary Banach space, provided that the initial point is right-scattered. Thus, a natural question arises, namely to establish sufficient conditions for the existence of local solutions of \eqref{eq_01_01} on a general time scale and in a general Banach space.

Some existence results for classical differential equations in Banach spaces can be found e.g. in \cite{Ambrosetti, Banas_Goebel, Li, Monch_vonHarten, Oberta}. Other special cases of dynamic equations, namely difference and q-difference equations, are studied e.g. in \cite{Agarwal_Thompson_Tisdell} and \cite{Silindir_Soyoglu}, respectively. Continuous dependence of solutions of dynamic equations on their domain, with examples of dynamic equations on the Cantor set, is studied in \cite{Cichon_Yantir}. For the case of general dynamic equations on time scales in Banach spaces, various existence results for the IVP \eqref{eq_01_01} can be found in \cite{Ahmed_Hazarika, Cichon_Kubiaczyk_SikorskaNowak_Yantir_2012, Cichon_Kubiaczyk_SikorskaNowak_Yantir_2009, Kubiaczyk_SikorskaNowak, SikorskaNowak, Tikare_Bohner_Azarika_Agarwal}. Note that, apart from the existence of classical solutions, also the existence of Carathéodory solutions, weak solutions on infinite time scales, and pseudosolutions for fractional dynamic equations are studied in these papers. Further results in this area include \cite{Yantir_Kubiaczyk_SikorskaNowak_2015, Yantir_Kubiaczyk_SikorskaNowak_2013}, where the question of existence of both classical and Carathéodory solutions of Sturm-Liouville dynamic equations on time scales in Banach spaces is addressed.

In our main result, \theoremRef{thm_08}, the requirements imposed on $f$ are of a different type (as far as we are concerned, neither less, nor more general) compared to the existing literature. Also, the comparison condition in our main result, which utilises the notion of a Kamke $\Delta$-function, is new when compared to the existing literature. Moreover, \theoremRef{thm_08} generalises a similar result from \cite{Banas_Goebel}, by which it was inspired, in the sense that for $\T=\R$, the existence result from \cite{Banas_Goebel} is a special case of our main theorem. Finally, our main result is developed directly for measures of noncompactness defined axiomatically, and thus is not restricted to the use of a specific, e.g. Kuratowski, measure of noncompactness, as is often the case in the existing literature. However, note that the use of the axiomatic theory of measures of (weak) noncompactness in connection with the theory of dynamic equations on time scales has been also mentioned e.g. in \cite{Cichon_Kubiaczyk_SikorskaNowak_Yantir_2012, Cichon_Kubiaczyk_SikorskaNowak_Yantir_2009, Yantir_Kubiaczyk_SikorskaNowak_2015}.

\sectionRef{sec_02} provides an introduction to the theory of calculus on time scales, the theory of measures of noncompactness, and it also introduces some new concepts. Firstly, the notion of a Kamke $\Delta$-function introduced here is completely new in the literature. It plays the role of a comparison function, and it is an extension of the Kamke function defined in \cite{Banas_Goebel, Li} (where such a notion is defined for the case of $\T=\R$) to the case of an arbitrary time scale. Secondly, a new concept of a uniform rd-equicontinuity for a family of functions is defined, which generalises the concept of the uniform equicontinuity. It is a new notion in the theory of time scales, and its definition is based on an equivalent characteristics for rd-continuous functions (see \lemmaRef{lm_01}), which is also a new result. Furthermore, it is also discussed in this section that a continuous solution of the corresponding Volterra-type integral equation is not necessarily a local solution of the IVP \eqref{eq_01_01} (see \exampleRef{ex_03} for a counterexample). As far as we are aware, such a discussion is unique in the existing literature. Finally, note that some of the auxiliary results, which are proved in this section, are new in the context of time scales. For more details, see \sectionRef{sec_05}.

In \sectionRef{sec_03} we prove our main existence theorem. The proof follows the main ideas of the proof of Theorem $13.3.1$ in \cite{Banas_Goebel}, which is a similar result for classical differential equations in Banach spaces. Also, note that a similar result was proved in \cite{Oberta} in the context of fractional differential equations in Banach spaces. Even though the general outline of the proof could be reused, the extension of those results to the context of dynamic equations on time scales requires establishment of various results from the theory of calculus on time scales, most of which are proved in \sectionRef{sec_02}. Also, numerous modifications are in place due to the use of rd-continuity instead of \enquote{classical} continuity, as well as due to other aspects connected to time scales.

In \sectionRef{sec_04} we provide an application of our main result in the Banach space $c_0$. We study a countable system of dynamic equations on a general time scale, which arises from semi-discretisation of the following parabolic partial dynamic equation
\begin{equation}\label{eq_01_02}
	\pDelta{u}{t}u\br{t,x}=\diffp[2]{u}{x}\br{t,x}+F\br{t,x}.
\end{equation}
In particular, we provide sufficient conditions for the existence of a local solution of the dynamic equation, corresponding to such a countable system, in the space $c_0$. Note that diffusion-type partial dynamic equations on a discrete space domain, which are similar to \eqref{eq_01_02}, are studied also in \cite{Slavik_Stehlik_2015, Slavik_Stehlik_2014, Slavik_Stehlik_Volek}.

\section{Auxiliary Results}\label{sec_02}
Let $\E=\br{\E,\norm{\cdot}_{\E}}$ denote a real Banach space. The neutral element of $\E$ will be denoted by $0_{\E}$. The closed (open) ball centred at $x_0\in\E$ with radius $r>0$ will be denoted by $\ball{\E}{x_0}{r}$ ($\openball{\E}{x_0}{r}$). For $X\subseteq{\E}$, let $\cl{X}$ denote the closure of $X$. Similarly, let $\conv{X}$ denote the convex hull of $X$. Recall that $\cl{\conv{X}}$, the closed convex hull of $X$, is the smallest closed and convex set containing $X$.

\subsection{Introduction to calculus on time scales}
In this subsection we introduce the notion of time scales, $\Delta$-derivative and Cauchy $\Delta$-integral, together with their basic properties. For more details see \cite{Hilger}, where the theory is developed for abstract-valued functions directly, or \cite{Bohner_Advances, Bohner_Dynamic}, where the theory is developed mostly for real-valued functions.

\textit{Time scale}, denoted by $\T$, is an arbitrary non-empty closed subset of real numbers. The \textit{forward jump operator} $\sigma\st\T\to\T$ is defined as $\sigma\br{t}\defeq\inf\cbr{s\in\T\st{s}>t}$, where we define $\inf\emptyset\defeq\sup\T$ (i.e. if a time scale has the maximum $M$, then $\sigma\br{M}=M$). Similarly, the \textit{backward jump operator} $\rho\st\T\to\T$ is defined as $\rho\br{t}\defeq\sup\cbr{s\in\T\st{s}<t}$, where we define $\sup\emptyset\defeq\inf\T$. We say that $t\in\T$ is \textit{right-dense} (\textit{left-dense}), if $\sigma\br{t}=t$ ($\rho\br{t}=t$). We say that $t\in\T$ is \textit{right-scattered} (\textit{left-scattered}), if $\sigma\br{t}>t$ ($\rho\br{t}<t$). If $\T$ has a left-scattered maximum $M$, we define $\Tk\defeq\T\setminus\cbr{M}$, otherwise we define $\Tk\defeq\T$.

We define $\sbrt{a,b}\defeq\sbr{a,b}\cap\T$. Moreover, when using the notation $\sbrt{a,b}$, we always assume that $a,b\in\T$. Similarly, we define $\brt{a,b}, \lsrbrt{a,b}$ and $\lbrsbrt{a,b}$ (whilst we always assume that $a,b\in\T$). To denote the intersection of a real interval and a time scale (with the endpoints not necessarily being part of $\T$), we write $\sbr{a,b}\cap\T$ (similarly for open and half-closed real intervals). Finally, for $t\in\T$ and $\delta>0$, we define $\nbrt{\delta}{t}\defeq\br{t-\delta,t+\delta}\cap\T$.

Further on, we assume that $\T$ is equipped with the standard topology inherited from $\R$. In particular, note that $u\st\T\to\E$ is trivially right-continuous (left-continuous) at every right-scattered (left-scattered) point.

\begin{definition}
	We say that $u\st\T\to\E$ is \textit{$\Delta$-differentiable} at $t\in\T$, if there exists $u^{\Delta}\br{t}\in\E$ such that
	\begin{equation*}
		\begin{aligned}
			&\br{\forall\eps>0}\br{\exists\delta>0}\br{\forall{s}\in\nbrt{\delta}{t}}\st\\
			&\qquad\br{\norm{u\br{\sigma\br{t}}-u\br{s}-\br{\sigma\br{t}-s}\cdot{u}^{\Delta}\br{t}}_{\E}<\eps\cdot\abs{\sigma\br{t}-s}}.
		\end{aligned}
	\end{equation*}
	In such a case, we call $u^{\Delta}\br{t}$ a \textit{$\Delta$-derivative} of $u$ at $t$.
\end{definition}

Note that if $u$ is $\Delta$-differentiable at $t\in\Tk$, then $u$ is continuous at $t$ and the value $u^{\Delta}\br{t}$ is determined uniquely. On the other hand, if $t\in\T\setminus\Tk$ (i.e. if $t$ is the left-scattered maximum of $\T$, provided that such a maximum exists), then it is rather easy to show that $u^{\Delta}\br{t}$ is not determined uniquely. Moreover, in such a case, every $x\in\E$ is actually a $\Delta$-derivative of $u\st\T\to\E$ at $t\in\T\setminus\Tk$. Finally, note that for $t\in\T\setminus\Tk$, both $u$ and $u^{\Delta}$ are trivially continuous at $t$ (since such a $t$ must be the left-scattered maximum of $\T$), regardless of the choice of $u^{\Delta}\br{t}$.

If $u$ is continuous at a right-scattered point $t\in\Tk$, then $u$ is $\Delta$-differentiable at $t$ and it holds that
\begin{equation}\label{eq_02_01}
	u^{\Delta}\br{t}=\frac{u\br{\sigma\br{t}}-u\br{t}}{\sigma\br{t}-t}.
\end{equation}
If $t\in\Tk$ is right-dense, then $u$ is $\Delta$-differentiable at $t$ if and only if the following limit exists (and is finite)
\begin{equation}\label{eq_02_02}
	\lim_{s\to{t},s\in\T}\frac{u\br{s}-u\br{t}}{s-t}.
\end{equation}
In such a case, $u^{\Delta}\br{t}$ is equal to that limit.

\begin{definition}\label{def_02}
	We say that $u\st\T\to\E$ is \textit{rd-continuous}, if $u$ is continuous at all the right-dense points and the left-sided limits exist (and are finite) at all the left-dense points.
\end{definition}

\begin{remark}\label{rem_01}
	Suppose that $\T$ has the maximum $M$. By the definition of $\sigma\br{\cdot}$, $\sigma\br{M}=M$ (i.e. $M$ is right-dense). Moreover, throughout this paper, whenever we work on some time scale interval $\sbrt{a,b}$, we always assume that $\sigma\br{b}=b$ (i.e. as if the whole time scale was equal to $\sbrt{a,b}$). In particular, \definitionRef{def_02} implies that every rd-continuous function defined on $\sbrt{a,b}$ must be continuous at $b$ (since by our previous convention, $b$ is right-dense).
\end{remark}

Note that every rd-continuous function defined on $\sbrt{a,b}$ is bounded.

We denote by $\C{\sbrt{a,b}}{\E}$ the space of all the continuous functions $u\st\sbrt{a,b}\to\E$. Note that the space $\C{\sbrt{a,b}}{\E}$ is a Banach space when equipped with the supremum norm $\norm{u}_{\infty}\defeq\sup_{t\in\sbrt{a,b}}\norm{u\br{t}}_{\E}$. Convergence in the space $\C{\sbrt{a,b}}{\E}$ (i.e. the uniform convergence) will be denoted by \enquote{$\rightrightarrows$}. Similarly, we denote by $\Crd{\sbrt{a,b}}{\E}$ the space of all the rd-continuous functions $u\st\sbrt{a,b}\to\E$. Moreover, we denote by $\CrdOne{\sbrt{a,b}}{\E}$ the space of all the functions $u\st\sbrt{a,b}\to\E$ such that $u\in\C{\sbrt{a,b}}{\E}$, $u^{\Delta}\br{\cdot}$ exists on $\sbrt{a,b}$ and $u^{\Delta}\in\Crd{\sbrt{a,b}}{\E}$. Finally, note that in the case of $u\st\sbrt{a,b}\to\E$ and $b$ being left-scattered, the (existence and) rd-continuity of $u^{\Delta}\br{\cdot}$ on $\sbrt{a,b}$ does not depend on $u^{\Delta}\br{b}$ (which in such a case always exists, but is not determined uniquely), since $u^{\Delta}\br{\cdot}$ is trivially continuous at the left-scattered maximum $b$, regardless of the value $u^{\Delta}\br{b}$.

Let $X\subseteq\cbr{u\st\sbrt{a,b}\to\E}$. Recall that the family of functions $X$ is said to be equibounded, if
\begin{equation*}
	\br{\exists{M}>0}\br{\forall{u}\in{X}}\br{\forall{t}\in\sbrt{a,b}}\st\br{\norm{u\br{t}}_{\E}\leq{M}}.
\end{equation*}
Moreover, recall that the family of functions $X$ is said to be uniformly equicontinuous, if
\begin{equation*}
	\br{\forall\eps>0}\br{\exists\delta>0}\br{\forall{u\in{X}}}\br{\forall{s,t\in\sbrt{a,b}}}\st\br{\abs{s-t}<\delta\implies\norm{u\br{s}-u\br{t}}_{\E}<\eps}.
\end{equation*}

\begin{definition}\label{def_03}
	Let $U\st\T\to\E$ be $\Delta$-differentiable on $\Tk$, and $u\st\Tk\to\E$ be such that $U^{\Delta}\br{\cdot}=u\br{\cdot}$ on $\Tk$. Then $U$ is called an \textit{antiderivative} of $u$. For $s,t\in\T$, we define the \textit{Cauchy $\Delta$-integral} of $u$ from $s$ to $t$ as
	\begin{equation*}
		\DeltaAint{s}{t}{u\br{\tau}\Deltatau}\defeq{U}\br{t}-U\br{s}.
	\end{equation*}
\end{definition}

In the case of $\E=\R$, we denote the Cauchy $\Delta$-integral simply by $\Deltaint{a}{b}{u\br{s}\Deltas}$. For a general Banach space $\E$, we use the notation $\DeltaAint{a}{b}{u\br{s}\Deltas}$ (with the prefix $\br{A}$, which stands for \enquote{abstract integral}).

\begin{remark}
	Notice that it is sufficient for $u$ from \definitionRef{def_03} to be defined only on $\Tk$. Yet, the Cauchy $\Delta$-integral of $u$ is well-defined on the whole $\T$. However, for convenience we often assume that $u$ is defined on $\T$.
\end{remark}

\begin{remark}\label{rem_03}
	Note that every rd-continuous function $u\st\sbrt{a,b}\to\E$ has an antiderivative. In particular, for a fixed $t_0\in\sbrt{a,b}$, an antiderivative of $u$ is of the form
	\begin{equation*}
		U\br{t}\defeq\DeltaAint{t_0}{t}{u\br{\tau}\Deltatau},\quad{t}\in\sbrt{a,b}.
	\end{equation*}
\end{remark}

Due to the fact that every $\Delta$-differentiable function is continuous, it immediately follows from \definitionRef{def_03} and \remarkRef{rem_03} that $\DeltaAint{a}{\cdot}{u\br{s}\Deltas}$ is well-defined and continuous on $\sbrt{a,b}$, for every $u\in\Crd{\sbrt{a,b}}{\E}$.

Moreover, for every $u\in\Crd{\T}{\E}$ and $v\in\Crd{\T}{\R}$, and for every $\sbrt{a,b}$, it holds that
\begin{equation}\label{eq_02_03}
	\br{\norm{u\br{t}}_{\E}\leq{v}\br{t},\forall{t}\in\lsrbrt{a,b}}\implies\norm{\DeltaAint{a}{b}{u\br{s}\Deltas}}_{\E}\leq\Deltaint{a}{b}{v\br{s}\Deltas}.
\end{equation}
Finally, note that for every $u\in\Crd{\sbrt{a,b}}{\E}$, it holds that
\begin{equation}\label{eq_02_04}
	\DeltaAint{a}{\sigma\br{a}}{u\br{s}\Deltas}=\br{\sigma\br{a}-a}\cdot{u}\br{a}.
\end{equation}

\subsection{Additional results from the theory of calculus on time scales}
In this subsection we prove some further auxiliary results, which will be needed later in this paper. Also, we provide an equivalent characteristics for rd-continuity, which will play an important role in the definition of the uniform rd-equicontinuity, which on the other hand will be utilised in our main theorem.

\begin{theorem}[\normalfont{Mean value theorem}]\label{thm_01}
	Let $u\in\Crd{\sbrt{a,b}}{\E}$. Then for all $t_0\in\sbrt{a,b}$ it holds that
	\begin{equation}\label{eq_02_05}
		\DeltaAint{a}{t_0}{u\br{s}\Deltas}\in\br{t_0-a}\cdot\cl{\conv{\cbr{u\br{s}\st{s}\in\sbrt{a,t_0}}}}.
	\end{equation}
\end{theorem}
\begin{proof}
	Firstly, note that according to \remarkRef{rem_03}, the integral in \eqref{eq_02_05} is well-defined.

	Let $t_0\in\sbrt{a,b}$ be arbitrary. Define $K$ as
	\begin{equation*}
		K\defeq\cl{\conv{\cbr{u\br{s}\st{s}\in\sbrt{a,t_0}}}}.
	\end{equation*}
	It can be shown (see also the proof of Lemma $1$ in \cite{Oberta}), using one of the consequences of the Hahn-Banach theorem, that
	\begin{equation}\label{eq_02_06}
		K=\bigcap_{\lambda\in\R,h\in\Est\,\st\cbr{x\in\E\,\st{h}\br{x}\leq\lambda}\supseteq{K}}\cbr{x\in\E\st{h}\br{x}\leq\lambda},
	\end{equation}
	where $\Est$ denotes the dual space of $\E$ (i.e. the space of all the bounded linear functionals on $\E$).
	Also, note that the intersection on the RHS of \eqref{eq_02_06} is over a non-empty set, as for $\lambda=0$ and $h=0_{\Est}$, it holds that $\cbr{x\in\E\st{h}\br{x}\leq\lambda}=\E\supseteq{K}$.
	
	Let $\lambda_0\in\R$ and $h_0\in\Est$ be arbitrary such that $K\subseteq\cbr{x\in{\E}\st{h}_0\br{x}\leq\lambda_0}$. According to \eqref{eq_02_06}, in order to prove \eqref{eq_02_05}, it suffices to show that
	\begin{equation*}
		h_0\br{\DeltaAint{a}{t_0}{u\br{s}\Deltas}}\leq\br{t_0-a}\cdot\lambda_0.
	\end{equation*}
	
	Define $g\st\sbrt{a,b}\to\E$ as $g\br{\cdot}\defeq\DeltaAint{a}{\cdot}{u\br{s}\Deltas}$. From \definitionRef{def_03}, it immediately follows that $g\br{\cdot}$ is $\Delta$-differentiable on $\sbrtk{a,b}$ and that $g^{\Delta}\br{\cdot}=u\br{\cdot}$ on $\sbrtk{a,b}$. Next, define $\varphi\st\sbrt{a,b}\to\R$ as $\varphi\br{\cdot}\defeq{h}_0\br{g\br{\cdot}}$. Using the linearity and continuity of $h_0\br{\cdot}$, it is rather easy to show that $\varphi\br{\cdot}$ is $\Delta$-differentiable on $\sbrtk{a,b}$ and that $\varphi^{\Delta}\br{\cdot}=h_0\br{g^{\Delta}\br{\cdot}}$ on $\sbrtk{a,b}$. Thus, it holds that
	\begin{equation}\label{eq_02_07}
		\begin{aligned}
			&h_0\br{\DeltaAint{a}{t_0}{u\br{s}\Deltas}}=h_0\br{g\br{t_0}}\eqText{g\br{a}=0}h_0\br{g\br{t_0}-g\br{a}}\\
			&\qquad\eqText{h_0\br{\cdot}\text{ is linear}}h_0\br{g\br{t_0}}-h_0\br{g\br{a}}=\varphi\br{t_0}-\varphi\br{a}.
		\end{aligned}
	\end{equation}
	By the mean value theorem for real-valued functions on time scales (see Theorem $1.14$ in \cite{Bohner_Advances}), there exists $\xi\in\lsrbrt{a,t_0}$ such that $\varphi\br{t_0}-\varphi\br{a}\leq\br{t_0-a}\cdot\varphi^{\Delta}\br{\xi}$. Thus, for such a $\xi$, \eqref{eq_02_07} yields that
	\begin{equation*}
		\begin{aligned}
			&h_0\br{\DeltaAint{a}{t_0}{u\br{s}\Deltas}}\leq\br{t_0-a}\cdot\varphi^{\Delta}\br{\xi}=\br{t_0-a}\cdot{h}_0\br{g^{\Delta}\br{\xi}}=\br{t_0-a}\cdot{h}_0\br{u\br{\xi}}\\
			&\qquad\leqText{u\br{\xi}\in{K}\subseteq\cbr{x\in\E\,\st{h}_0\br{x}\leq\lambda_0}}\br{t_0-a}\cdot\lambda_0.
		\end{aligned}
	\end{equation*}
\end{proof}

\begin{theorem}[\normalfont{Arzelà-Ascoli theorem, \cite{Munkres}, Theorem $47.1$}]\label{thm_02}
	Let $X\subseteq\C{\sbrt{a,b}}{\E}$. Then the set $X$ is relatively compact in $\C{\sbrt{a,b}}{\E}$ if and only if the family of functions $X$ is uniformly equicontinuous on $\sbrt{a,b}$ and the set $X\br{t}$ is relatively compact in $\E$, for every $t\in\sbrt{a,b}$.
\end{theorem}

\begin{definition}\label{def_04}
	We say that a family of functions $X\subseteq\cbr{u\st\sbrt{a,b}\to\E}$ is \textit{uniformly rd-equicontinuous} on $\sbrt{a,b}$, if
	\begin{equation}\label{eq_02_08}
		\begin{aligned}
			&\br{\forall\eps>0}\br{\exists\delta_1,\dots,\delta_n>0}\\
			&\qquad\br{\exists\tau_1,\dots,\tau_n\in\sbrt{a,b}\st{a}=\tau_1<\dots<\tau_n=b\text{ and }\sbrt{a,b}\subseteq\bigcup_{i=1}^n\nbrt{\delta_i}{\tau_i}}\\
			&\qquad\br{\forall{u}\in{X}}\br{\forall{j}=1,\dots,n}\br{\forall{s},t\in\nbrt{\delta_j}{\tau_j}}\st\\
			&\qquad\br{\br{\tau_j\text{ is right-dense or }s,t<\tau_j\text{ or }s,t\geq\tau_j}\implies\norm{u\br{s}-u\br{t}}_{\E}<\eps}.
		\end{aligned}
	\end{equation}
\end{definition}

\begin{remark}\label{rem_04}
	In particular, as we will show in the next lemma, \definitionRef{def_04} implies that for every uniformly rd-equicontinuous family $X$ on $\sbrt{a,b}$, it holds that $X\subseteq\Crd{\sbrt{a,b}}{\E}$.
\end{remark}

\begin{lemma}\label{lm_01}
	Let $u\st\sbrt{a,b}\to\E$. Then $u$ is rd-continuous on $\sbrt{a,b}$ if and only if the following condition holds
	\begin{equation}\label{eq_02_09}
		\begin{aligned}
			&\br{\forall\eps>0}\br{\exists\delta_1,\dots,\delta_n>0}\\
			&\qquad\br{\exists\tau_1,\dots,\tau_n\in\sbrt{a,b}\st{a}=\tau_1<\dots<\tau_n=b\text{ and }\sbrt{a,b}\subseteq\bigcup_{i=1}^n\nbrt{\delta_i}{\tau_i}}\\
			&\qquad\br{\forall{j}=1,\dots,n}\br{\forall{s},t\in\nbrt{\delta_j}{\tau_j}}\st\\
			&\qquad\br{\br{\tau_j\text{ is right-dense or }s,t<\tau_j\text{ or }s,t\geq\tau_j}\implies\norm{u\br{s}-u\br{t}}_{\E}<\eps}.
		\end{aligned}
	\end{equation}
\end{lemma}
\begin{proof}
	It can be easily shown, using the \enquote{$\eps\text{-}\delta$ language}, that \definitionRef{def_02} is equivalent to
	\begin{equation}\label{eq_02_10}
		\begin{aligned}
			&\br{\forall\eps>0}\br{\forall\tau\in\sbrt{a,b}}\br{\exists\delta_{\tau}>0}\br{\forall{s},t\in\nbrt{\delta_{\tau}}{\tau}}\st\\
			&\qquad\br{\br{\tau\text{ is right-dense or }s,t<\tau}\implies\norm{u\br{s}-u\br{t}}_{\E}<\eps}.
		\end{aligned}
	\end{equation}
	If $\tau\in\sbrt{a,b}$ is right-scattered, then the corresponding $\delta_{\tau}$ from \eqref{eq_02_10} can always be chosen as $0<\delta_{\tau}<\sigma\br{\tau}-\tau$ (so that $\lsrbr{\tau,\tau+\delta_{\tau}}\cap\T=\cbr{\tau}$). Hence, we readily obtain that \eqref{eq_02_10}, and thus also \definitionRef{def_02}, is equivalent to (notice the additional \enquote{or $s,t\geq\tau$} part) 
	\begin{equation}\label{eq_02_11}
		\begin{aligned}
			&\br{\forall\eps>0}\br{\forall\tau\in\sbrt{a,b}}\br{\exists\delta_{\tau}>0}\br{\forall{s},t\in\nbrt{\delta_{\tau}}{\tau}}\st\\
			&\qquad\br{\br{\tau\text{ is right-dense or }s,t<\tau\text{ or }s,t\geq\tau}\implies\norm{u\br{s}-u\br{t}}_{\E}<\eps}.
		\end{aligned}
	\end{equation}
	
	\enquote{$\implies$} Assume that $u\in\Crd{\sbrt{a,b}}{\E}$. Let $\eps>0$ be arbitrary. Consider the corresponding $\cbr{\delta_{\tau}}_{\tau\in\sbrt{a,b}}$ from \eqref{eq_02_11}. Obviously, $\bigcup_{\tau\in\sbrt{a,b}}\nbrt{\delta_{\tau}}{\tau}$ forms an open cover of the compact set $\sbrt{a,b}$. Thus, by the topological definition of compactness, there exists a finite subcover. Without loss of generality, let $a=\tau_1<\dots<\tau_n=b$ be such that $\bigcup_{i=1}^n\nbrt{\delta_i}{\tau_i}$ is the corresponding finite subcover of $\sbrt{a,b}$ (where for simplicity we write $\delta_i$ instead of $\delta_{\tau_i}$). This immediately yields that \eqref{eq_02_11} implies \eqref{eq_02_09}.
	
	\enquote{$\impliedby$} Assume that $u$ satisfies \eqref{eq_02_09}. It can be easily shown that for every $\eps>0$, and for every $\tau\in\sbrt{a,b}$, there exists $\delta_{\tau}>0$ such that the conclusion of \eqref{eq_02_10} holds.
\end{proof}

\begin{remark}
	One might wonder whether the additional \enquote{or $s,t\geq\tau$} part from \eqref{eq_02_11} is needed in \eqref{eq_02_09}. Consider the following modification of \eqref{eq_02_09} without that additional part (compare also with \eqref{eq_02_10})
	\begin{equation}\label{eq_02_12}
		\begin{aligned}
			&\br{\forall\eps>0}\br{\exists\delta_1,\dots,\delta_n>0}\\
			&\qquad\br{\exists\tau_1,\dots,\tau_n\in\sbrt{a,b}\st{a}=\tau_1<\dots<\tau_n=b\text{ and }\sbrt{a,b}\subseteq\bigcup_{i=1}^n\nbrt{\delta_i}{\tau_i}}\\
			&\qquad\br{\forall{j}=1,\dots,n}\br{\forall{s},t\in\nbrt{\delta_j}{\tau_j}}\st\\
			&\qquad\br{\br{\tau_j\text{ is right-dense or }s,t<\tau_j}\implies\norm{u\br{s}-u\br{t}}_{\E}<\eps}.
		\end{aligned}
	\end{equation}
	Obviously, \lemmaRef{lm_01} implies that every rd-continuous function satisfies \eqref{eq_02_12} as well. However, the opposite direction does not necessarily hold. This is due to the fact, roughly speaking, that \enquote{$\nbrt{\delta_i}{\tau_i}$ for some right-scattered $\tau_i$ can be so large that $\lsrbr{\tau_i,\tau_i+\delta_i}\cap\T\neq\cbr{\tau_i}$, but no assumptions are put on the function values at the points inside the right neighbourhood of $\tau_i$}. As a counterexample, see \exampleRef{ex_01}.
\end{remark}

\begin{example}\label{ex_01}
	Consider $\E=\R$, $\T\defeq\sbr{0,\frac{1}{2}}\cup\sbr{\frac{3}{4},1}$, and $u\st\T\to\R$ defined as $u\br{\frac{3}{4}}\defeq1$, and $u\br{t}\defeq0$, for $t\in\sbr{0,\frac{1}{2}}\cup\lbrsbr{\frac{3}{4},1}$. Apparently, $u$ is not rd-continuous, as it is not continuous at the right-dense point $\frac{3}{4}$. However, $u$ satisfies \eqref{eq_02_12}, since for every $\eps>0$, and $\tau_1=0,\tau_2=\frac{1}{2},\tau_3=1,\delta_1=\frac{1}{8},\delta_2=\frac{1}{2},\delta_3=\frac{1}{8}$, the conclusion of \eqref{eq_02_12} holds. Nevertheless, such a partition cannot be used in the case of \eqref{eq_02_09}, since \eqref{eq_02_09} would require also $\abs{u\br{\frac{1}{2}}-u\br{\frac{3}{4}}}<\eps$ (due to the fact that $\frac{1}{2},\frac{3}{4}\in\nbrt{\delta_2}{\tau_2}$ and $\frac{1}{2},\frac{3}{4}\geq\tau_2$), which is not true for $\eps<1$.
\end{example}

\begin{lemma}\label{lm_02}
	Let $X\subseteq\Crd{\sbrt{a,b}}{\E}$ be a uniformly rd-equicontinuous family on $\sbrt{a,b}$. Then for all $t_0\in\sbrt{a,b}$ it holds that
	\begin{equation}\label{eq_02_13}
		\begin{aligned}
			&\br{\forall\eps>0}\br{\exists\xi_1,\dots,\xi_m\in\sbrt{a,t_0}\st{a}=\xi_1<\dots<\xi_m=t_0}\\
			&\qquad\br{\forall{u}\in{X}}\br{\forall{j}=1,\dots,m-1}\br{\forall{s},t\in\lsrbrt{\xi_j,\xi_{j+1}}}\st\br{\norm{u\br{s}-u\br{t}}_{\E}<\eps}.
		\end{aligned}
	\end{equation}
\end{lemma}
\begin{proof}
	Without loss of generality, it is enough to prove the lemma for $t_0=b$. Let $\eps>0$ be arbitrary. Consider the corresponding $\delta_1,\dots,\delta_n>0$ and $\tau_1,\dots,\tau_n\in\sbrt{a,b}$ from \eqref{eq_02_08}. It is easy to see that \eqref{eq_02_08} implies that for every $u\in{X}$, and for all $k=1,\dots,n-1$, the following conclusions hold
	\begin{subequations}
		\begin{align}
			\br{\forall{s},t\in\lsrbr{\tau_k,\tau_k+\delta_k}\cap\T}\st&\br{\norm{u\br{s}-u\br{t}}_{\E}<\eps};\label{eq_02_14_01}\\
			\br{\forall{s},t\in\br{\tau_{k+1}-\delta_{k+1},\tau_{k+1}}\cap\T}\st&\br{\norm{u\br{s}-u\br{t}}_{\E}<\eps}.\label{eq_02_14_02}
		\end{align}
	\end{subequations}
	
	Now, define the finite partition $\cbr{\xi_i}_{i=1}^m$ of $\sbrt{a,b}$ inductively as follows
	\begin{enumerate}[I.]
		\item\label{lm_02_def_01} Define
		\begin{equation*}
			\xi_1\defeq\tau_1.
		\end{equation*}
		\item\label{lm_02_def_02} For all $j\geq1$, if $\xi_j=\tau_k$, for some $1\leq{k}\leq{n-1}$, define
		\begin{subequations}
			\begin{align}
				L_k&\defeq\sbrt{\tau_k,\tau_{k+1}}\cap\lsrbr{\tau_k,\tau_k+\delta_k};\label{eq_02_15_01}\\
				R_k&\defeq\sbrt{\tau_k,\tau_{k+1}}\cap\lbrsbr{\tau_{k+1}-\delta_{k+1},\tau_{k+1}}.\label{eq_02_15_02}
			\end{align}
		\end{subequations}
		Note that both $L_k$ and $R_k$ are non-empty, since $\tau_k\in{L}_k$ and $\tau_{k+1}\in{R}_k$.
		\begin{enumerate}[1)]
			\item If $L_k\cap{R}_k\neq\emptyset$, let $\eta_k\in{L}_k\cap{R}_k$ be arbitrary and define
			\begin{equation*}
				\begin{aligned}
					\xi_{j+1}&\defeq\eta_k;\\
					\xi_{j+2}&\defeq\tau_{k+1}.
				\end{aligned}
			\end{equation*}
			In particular, note that $\xi_{j}\leq\xi_{j+1}\leq\xi_{j+2}$.
			\item If $L_k\cap{R}_k=\emptyset$, define
			\begin{subequations}
				\begin{align}
					\xi_{j+1}&\defeq\sup\cbr{s\st{s}\in{L}_k};\label{eq_02_16_01}\\
					\xi_{j+2}&\defeq\inf\cbr{s\st{s}\in{R}_k};\label{eq_02_16_02}\\
					\xi_{j+3}&\defeq\tau_{k+1}.
				\end{align}
			\end{subequations}
			In particular, note that $\xi_{j}\leq\xi_{j+1}$ and $\xi_{j+2}\leq\xi_{j+3}$. Moreover, it holds also that $\xi_{j+1}\leq\xi_{j+2}$ (otherwise, it would be a contradiction with the fact that $L_k\cap{R}_k=\emptyset$).
		\end{enumerate}
	\end{enumerate}
	Note that the definition of $\cbr{\xi_i}_{i=1}^m$ is correct. Indeed, in \partRef{lm_02_def_02} we cover the case when $\xi_j=\tau_k$, for some $1\leq{k}\leq{n-1}$. For $j=1$, we have that $\xi_1=\tau_1$ (according to \partRef{lm_02_def_01}), thus \partRef{lm_02_def_02} occurs at least once. Moreover, whenever $\xi_j=\tau_k$, for some $1\leq{k}\leq{n-1}$, we have that either $\xi_{j+2}=\tau_{k+1}$, or $\xi_{j+3}=\tau_{k+1}$ (see the individual subcases of \partRef{lm_02_def_02}). In particular, whenever $\xi_j=\tau_k$ occurs in \partRef{lm_02_def_02}, we obtain another pair of the form $\xi_{j^{\prime}}=\tau_{k+1}$ (where $\xi_{j^{\prime}}$ is either $\xi_{j+2}$, or $\xi_{j+3}$). Firstly, this implies that the number of occurrences of \partRef{lm_02_def_02} is finite. Secondly, this implies that \partRef{lm_02_def_02} is sufficient and there is no need for another part of the form \enquote{for all $j\geq1$, if $\xi_j\neq\tau_k$, for all $1\leq{k}\leq{n-1}$}. Moreover, note that $\cbr{\tau_i}_{i=1}^n\subseteq\cbr{\xi_i}_{i=1}^m\subseteq\sbrt{a,b}$, and that $a=\tau_1=\xi_1\leq\xi_2\leq\dots\leq\xi_{m-1}\leq\xi_m=\tau_n=b$.
	
	Now, we verify that \eqref{eq_02_13} holds (up to the fact that $\xi_1\leq\dots\leq\xi_m$ are not necessarily distinct). Let $u\in{X}$ be arbitrary. Due to the definition of $\cbr{\xi_i}_{i=1}^m$, it is enough to verify that the conclusion of \eqref{eq_02_13} holds for all $j\geq1$ such that $\xi_j=\tau_k$, for some $1\leq{k}\leq{n-1}$, and for both the subcases from \partRef{lm_02_def_02} of the definition of $\cbr{\xi_i}_{i=1}^m$. Let $j_0\geq1$ be arbitrary such that $\xi_{j_0}=\tau_k$, for some $1\leq{k}\leq{n-1}$. Consider the individual subcases from \partRef{lm_02_def_02}
	\begin{enumerate}[1)]
		\item Since $\eta_k\in{L}_k$, it follows from \eqref{eq_02_15_01} that $\tau_k\leq\eta_k<\tau_k+\delta_k$. Keeping in mind that $\xi_{j_0}=\tau_k$ and $\xi_{j_0+1}=\eta_k$, the conclusion of \eqref{eq_02_13} is true for $j=j_0$ by \eqref{eq_02_14_01}. Similarly, since $\eta_k\in{R}_k$, it follows from \eqref{eq_02_15_02} that $\tau_{k+1}-\delta_{k+1}<\eta_k\leq\tau_{k+1}$. Keeping in mind that $\xi_{j_0+1}=\eta_k$ and $\xi_{j_0+2}=\tau_{k+1}$, the conclusion of \eqref{eq_02_13} is true for $j=j_0+1$ by \eqref{eq_02_14_02}.
		\item By \eqref{eq_02_16_01} and \eqref{eq_02_15_01}, it follows that $\xi_{j_0+1}\leq\tau_k+\delta_k$. Keeping in mind that $\xi_{j_0}=\tau_k$, the conclusion of \eqref{eq_02_13} is true for $j=j_0$ by \eqref{eq_02_14_01}.
		
		Similarly, from \eqref{eq_02_16_02} and \eqref{eq_02_15_02}, it follows that $\xi_{j_0+2}\geq\tau_{k+1}-\delta_{k+1}$. We claim that the case \enquote{=} cannot occur. Indeed, suppose that $\xi_{j_0+2}=\tau_{k+1}-\delta_{k+1}$. According to \eqref{eq_02_15_02}, this means that $\xi_{j_0+2}\notin{R}_k$, and thus $\xi_{j_0+2}$ is right-dense (otherwise, it would be a contradiction with \eqref{eq_02_16_02}). Moreover, $\xi_{j_0+2}\notin{L}_k$ (otherwise, according to \eqref{eq_02_15_01}, it would mean that $\xi_{j_0+2}<\tau_k+\delta_k$, which due to \eqref{eq_02_16_02} and the fact that $\xi_{j_0+2}$ is right-dense would mean that $L_k\cap{R}_k\neq\emptyset$, which is a contradiction). Since $\xi_{j_0+2}\notin{L}_k$, \eqref{eq_02_15_01} yields that $\xi_{j_0+2}\geq\tau_k+\delta_k$. Together with the assumption that $\xi_{j_0+2}=\tau_{k+1}-\delta_{k+1}$, we obtain a contradiction with the fact that $\xi_{j_0+2}\in\sbrt{a,b}\subseteq\bigcup_{i=1}^n\nbrt{\delta_i}{\tau_i}$. Thus, it must hold that $\xi_{j_0+2}>\tau_{k+1}-\delta_{k+1}$, which implies that the conclusion of \eqref{eq_02_13} is true for $j=j_0+2$ by \eqref{eq_02_14_02} (recall that $\xi_{j_0+3}=\tau_{k+1}$).
		
		Finally, we claim that $\brt{\xi_{j_0+1},\xi_{j_0+2}}=\emptyset$. Indeed, suppose that there exists $\omega\in\T$ such that $\xi_{j_0+1}<\omega<\xi_{j_0+2}$. By \eqref{eq_02_16_01}, it follows that $\omega\notin{L}_k$. Thus, according to \eqref{eq_02_15_01}, $\omega\geq\tau_k+\delta_k$. Similarly, from \eqref{eq_02_16_02} and \eqref{eq_02_15_02}, it follows that $\omega\leq\tau_{k+1}-\delta_{k+1}$. Hence, we obtain a contradiction with the fact that $\omega\in\sbrt{a,b}\subseteq\bigcup_{i=1}^n\nbrt{\delta_i}{\tau_i}$. Thus, it must hold that $\brt{\xi_{j_0+1},\xi_{j_0+2}}=\emptyset$, which means that the conclusion of \eqref{eq_02_13} is trivially true for $j=j_0+1$.
	\end{enumerate}
	Finally, removing the duplicates from $\cbr{\xi_i}_{i=1}^m$, we obtain precisely \eqref{eq_02_13}.
\end{proof}

\subsection{Initial value problem for dynamic equations on time scales}\label{sec_02_03}
In this subsection we introduce an initial value problem (IVP) for dynamic equations on time scales in Banach spaces. Also, we discuss under what conditions does the equivalence between solutions to the IVP and the corresponding Volterra-type integral equation hold.

Let $u_0\in\E$ and $f\st\sbrt{a,b}\times\E\to\E$. Consider the following IVP
\begin{subequations}\label{eq_02_17}
	\begin{align}
		u^{\Delta}\br{t}&=f\br{t,u\br{t}},\quad{t}\in\sbrtk{a,b},\label{eq_02_17_01}\\
		u\br{a}&=u_0.\label{eq_02_17_02}
	\end{align}
\end{subequations}
By a local solution of \eqref{eq_02_17} on $\sbrt{a,c}\subseteq\sbrt{a,b}$, we mean $u\in\CrdOne{\sbrt{a,c}}{\E}$ such that \eqref{eq_02_17_01} holds for all $t\in\sbrtk{a,c}$ and the initial condition \eqref{eq_02_17_02} is satisfied. Under certain assumptions (specified in \theoremRef{thm_03}), solving \eqref{eq_02_17} is equivalent to solving the following Volterra-type integral equation
\begin{equation}\label{eq_02_18}
	u\br{t}=u_0+\DeltaAint{a}{t}{f\br{s,u\br{s}}\Deltas},\quad{t}\in\sbrt{a,b}.
\end{equation}

Before stating the theorem concerning the equivalence of solutions of \eqref{eq_02_17} and \eqref{eq_02_18}, we need to discuss restrictions of rd-continuous functions to time scale subintervals.

\begin{remark}\label{rem_06}
	Let $u\in\Crd{\sbrt{a,b}}{\E}$ and $\sbrt{a,c}\subset\sbrt{a,b}$. Consider the restriction $u\arrowvert_{\sbrt{a,c}}$ of $u$ to $\sbrt{a,c}$. We are interested whether $u\arrowvert_{\sbrt{a,c}}$ is rd-continuous (according to \definitionRef{def_02} and \remarkRef{rem_01}) on $\sbrt{a,c}$. Obviously, $u\arrowvert_{\sbrt{a,c}}$ satisfies \definitionRef{def_02} for every $t\in\lsrbrt{a,c}$. However, according to our convention described in \remarkRef{rem_01}, in order to be rd-continuous on $\sbrt{a,c}$, $u\arrowvert_{\sbrt{a,c}}$ would need to be continuous at $c$ (since the point $c$ is always right-dense with respect to the time scale $\sbrt{a,c}$). If in the \textit{original} time scale $\sbrt{a,b}$, $c\in\sbrt{a,b}$ is right-scattered and left-dense, it may happen that $u$ is not continuous at that point. Thus, in such a case, $u\arrowvert_{\sbrt{a,c}}$ is not rd-continuous on $\sbrt{a,c}$, since it is not continuous at $c\in\sbrt{a,c}$, which in the \textit{new} time scale $\sbrt{a,c}$ is right-dense. However, if in the \textit{original} time scale $\sbrt{a,b}$, $c\in\sbrt{a,b}$ is right-dense or left-scattered, $u$ must be continuous at $c\in\sbrt{a,b}$, which means that $u\arrowvert_{\sbrt{a,c}}$ is continuous at $c\in\sbrt{a,c}$ as well (and thus $u\arrowvert_{\sbrt{a,c}}\in\Crd{\sbrt{a,c}}{\E}$).
	
	To sum up, if $c\in\sbrt{a,b}$ is right-dense or left-scattered, the restriction $u\arrowvert_{\sbrt{a,c}}$ is always rd-continuous on $\sbrt{a,c}$. However, rd-continuity of $u\arrowvert_{\sbrt{a,c}}$ is not guaranteed if $c\in\sbrt{a,b}$ is right-scattered and left-dense. As a counterexample, see \exampleRef{ex_02}.
\end{remark}

\begin{example}\label{ex_02}
	Consider $\E=\R$, $\T\defeq\sbr{0,1}\cup\sbr{2,3}$, and $u\st\T\to\R$ defined as $u\br{1}\defeq1$, and $u\br{t}\defeq0$, for $t\in\lsrbr{0,1}\cup\sbr{2,3}$. Obviously, $u$ is rd-continuous on $\T$. However, its restriction $u\arrowvert_{\sbr{0,1}}$ is not rd-continuous on $\sbr{0,1}$, since it is not continuous at $t=1$ (which is right-dense with respect to $\sbr{0,1}$).
\end{example}

However, as we will see in the next remark, subintervals of the form $\sbrt{a,\sigma\br{c}}$ preserve rd-continuity, regardless of the type of the point $c\in\sbrt{a,b}$.

\begin{remark}\label{rem_07}
	Let $u\in\Crd{\sbrt{a,b}}{\E}$. It is easy to verify that for every $c\in\sbrt{a,b}$, $\sigma\br{c}$ is always right-dense or left-scattered with respect to $\sbrt{a,b}$. Thus, according to \remarkRef{rem_06}, the restriction $u\arrowvert_{\sbrt{a,\sigma\br{c}}}$ of $u$ to $\sbrt{a,\sigma\br{c}}$ is guaranteed to be rd-continuous on $\sbrt{a,\sigma\br{c}}$, for every $c\in\sbrt{a,b}$.
\end{remark}

\begin{remark}\label{rem_08}
	Let $X\subseteq\Crd{\sbrt{a,b}}{\E}$ be a uniformly rd-equicontinuous family on $\sbrt{a,b}$. By similar arguments as in \remarkRef{rem_06} and \remarkRef{rem_07} (the points of the form $\sigma\br{\cdot}$ being right-dense or left-scattered), one can establish directly from \definitionRef{def_04} that the restriction of $X$ to $\sbrt{a,\sigma\br{c}}$ is uniformly rd-equicontinuous on $\sbrt{a,\sigma\br{c}}$, for every $c\in\sbrt{a,b}$.
\end{remark}

Now we are ready to state and prove the theorem regarding the equivalence of solutions of the IVP \eqref{eq_02_17} and the solutions of the Volterra-type integral equation \eqref{eq_02_18}. Even though such an equivalence is well-known and widely used in the literature, we believe that a technical, yet important, detail is often missing there, namely the fact that the restriction of an rd-continuous function to a time scale subinterval need not be rd-continuous (see \exampleRef{ex_02} for a counterexample). In \remarkRef{rem_09} we discuss what implications this might have on the relationship between the solutions of \eqref{eq_02_17} and \eqref{eq_02_18}. Then, in \theoremRef{thm_03} we provide such an equivalence theorem, although in a different form compared to such statements found in the literature. Finally, special cases of this general theorem, which might be more suitable for applications, are provided in \theoremRef{thm_04} and \theoremRef{thm_05}, respectively.

\begin{theorem}\label{thm_03}
	Consider the IVP \eqref{eq_02_17}. Let $u\in\C{\sbrt{a,b}}{\E}$ be such that $f\br{\cdot,u\br{\cdot}}\in\Crd{\sbrt{a,b}}{\E}$. Then for every $c\in\sbrt{a,b}$, $u\arrowvert_{\sbrt{a,\sigma\br{c}}}$ is a local solution of \eqref{eq_02_17} on $\sbrt{a,\sigma\br{c}}$ if and only if it satisfies \eqref{eq_02_18}, for all $t\in\sbrt{a,\sigma\br{c}}$.
\end{theorem}
\begin{proof}
	According to \remarkRef{rem_07}, it holds that $f\br{\cdot,u\br{\cdot}}\arrowvert_{\sbrt{a,\sigma\br{c}}}\in\Crd{\sbrt{a,\sigma\br{c}}}{\E}$, for every $c\in\sbrt{a,b}$. The theorem is then an immediate consequence of \definitionRef{def_03} and \remarkRef{rem_03}, respectively.
\end{proof}

\begin{remark}\label{rem_09}
	In \theoremRef{thm_03}, the assumption of the local interval having the right endpoint of the form $\sigma\br{\cdot}$ is crucial. Otherwise, the \enquote{$\impliedby$} direction in \theoremRef{thm_03} does not necessarily hold. As a counterexample, see \exampleRef{ex_03}. Note that the failure of \theoremRef{thm_03} there is due to the fact that, in \exampleRef{ex_03}, the equality $\br{\Deltaint{0}{\cdot}{f\arrowvert_{\sbr{0,1}}\br{s,0}\Deltas}}^{\Delta}\br{t}=f\arrowvert_{\sbr{0,1}}\br{t,0}$ holds only for $t\in\lsrbr{0,1}$ (but not for $t=1$). Finally, note that this is not a contradiction with \remarkRef{rem_03}, since $f\arrowvert_{\sbr{0,1}}\br{\cdot,0}\notin\Crd{\sbr{0,1}}{\R}$.
\end{remark}

\begin{example}\label{ex_03}
	Consider \eqref{eq_02_17} with $\E=\R$, $u_0\defeq0$, $\T\defeq\sbr{0,1}\cup\sbr{2,3}$, and $f\st\T\times\R\to\R$ defined as $f\br{1,x}\defeq1$, for $x\in\R$, and $f\br{t,x}\defeq0$, for $\br{t,x}\in\br{\lsrbr{0,1}\cup\sbr{2,3}}\times\R$ (compare with \exampleRef{ex_02}). We are interested in local solutions of \eqref{eq_02_17} on $\sbr{0,1}$. It is easy to see that for $u\equiv0$ on $\T$, it holds that $f\br{\cdot,0}\in\Crd{\T}{\R}$. However, the restriction of $f\br{\cdot,0}$ to $\sbr{0,1}$ is not rd-continuous (but rather so called \textit{regulated}). The notion of Cauchy $\Delta$-integral in \eqref{eq_02_18} for such functions is well-defined (although with a bit more general definition than \definitionRef{def_03}). In such a case, we would have that $\Deltaint{0}{t}{f\arrowvert_{\sbr{0,1}}\br{s,0}\Deltas}=0$, for all $t\in\sbr{0,1}$. Thus, $u\equiv0$ satisfies \eqref{eq_02_18}, for all $t\in\sbr{0,1}$. However, $u\equiv0$ satisfies \eqref{eq_02_17_01} only for $t\in\lsrbr{0,1}$, but not for $t=1$. Hence, $u\equiv0$ is a solution of \eqref{eq_02_18} on $\sbr{0,1}$, but it is not a local solution of \eqref{eq_02_17} on $\sbr{0,1}$.
\end{example}

\begin{remark}
	Note that in \theoremRef{thm_03}, it is implicitly assumed that $u\in\C{\sbrt{a,b}}{\E}$ is such that $f\br{\cdot,u\br{\cdot}}\in\Crd{\sbrt{a,b}}{\E}$. This is because, for a given $c\in\sbrt{a,b}$, we need to guarantee that $f\br{\cdot,u\br{\cdot}}\arrowvert_{\sbrt{a,\sigma\br{c}}}\in\Crd{\sbrt{a,\sigma\br{c}}}{\E}$. According to \remarkRef{rem_07}, this is guaranteed (for the intervals of the form $\sbrt{a,\sigma\br{\cdot}}$) by $f\br{\cdot,u\br{\cdot}}$ being rd-continuous on $\sbrt{a,b}$. Otherwise, if the restriction is not rd-continuous, the conclusion of \theoremRef{thm_03} might not necessarily hold, as was shown in \exampleRef{ex_03}. However, in applications it is often more convenient to have the assumptions on $f$ directly, without making further a priori requirements on $u$ (apart from $u$ being continuous). Such sufficient conditions are provided in the next theorems.
\end{remark}

\begin{theorem}\label{thm_04}
	Consider the IVP \eqref{eq_02_17}. Let $f$ be continuous on $\sbrt{a,b}\times\E$, and let $u\in\C{\sbrt{a,b}}{\E}$. Then for every $c\in\sbrt{a,b}$, $u\arrowvert_{\sbrt{a,\sigma\br{c}}}$ is a local solution of \eqref{eq_02_17} on $\sbrt{a,\sigma\br{c}}$ if and only if it satisfies \eqref{eq_02_18}, for all $t\in\sbrt{a,\sigma\br{c}}$.
\end{theorem}
\begin{proof}
	Since $f$ is continuous on $\sbrt{a,b}\times\E$ and $u\in\C{\sbrt{a,b}}{\E}$, it is well-known that $f\br{\cdot,u\br{\cdot}}\in\C{\sbrt{a,b}}{\E}$ as well. Thus, the claim follows from \theoremRef{thm_03}. 
\end{proof}

\begin{theorem}\label{thm_05}
	Consider the IVP \eqref{eq_02_17}. Let $f\st\sbrt{a,b}\times\ball{\E}{u_0}{\beta}\to\E$, for some $\beta>0$. Assume that
	\begin{enumerate}[(i)]
		\item\label{thm_05_assum_01} $f\br{\cdot,x}$ is rd-continuous on $\sbrt{a,b}$, for all $x\in\ball{\E}{u_0}{\beta}$.
		\item\label{thm_05_assum_02} The family $\cbr{f\br{t,\cdot}}_{t\in\sbrt{a,b}}$ is uniformly equicontinuous on $\ball{\E}{u_0}{\beta}$.
	\end{enumerate}
	Then for every $c\in\sbrt{a,b}$, $u\in\C{\sbrt{a,\sigma\br{c}}}{\E}$ such that $\cbr{u\br{t}\st{t}\in\sbrt{a,\sigma\br{c}}}\subseteq\ball{\E}{u_0}{\beta}$ is a local solution of \eqref{eq_02_17} on $\sbrt{a,\sigma\br{c}}$ if and only if it satisfies \eqref{eq_02_18}, for all $t\in\sbrt{a,\sigma\br{c}}$. 
\end{theorem}
\begin{proof}
	Let $c\in\sbrt{a,b}$ and $u\in\C{\sbrt{a,\sigma\br{c}}}{\E}$ such that $\cbr{u\br{t}\st{t}\in\sbrt{a,\sigma\br{c}}}\subseteq\ball{\E}{u_0}{\beta}$ be arbitrary. Define a continuous extension $\tilde{u}$ of $u$ to $\sbrt{a,b}$ as $\tilde{u}\br{t}\defeq{u}\br{t}$, for $t\in\sbrt{a,\sigma\br{c}}$, and $\tilde{u}\br{t}\defeq{u}\br{\sigma\br{c}}$, for $t\in\lbrsbrt{\sigma\br{c},b}$. In particular, note that $\tilde{u}$ is indeed continuous. Moreover, using \assumptionRef{thm_05_assum_01} and \assumptionRef{thm_05_assum_02}, together with the continuity of $\tilde{u}$ and the following triangle inequality
	\begin{equation}\label{eq_02_19}
		\norm{f\br{s,\tilde{u}\br{s}}-f\br{t,\tilde{u}\br{t}}}_{\E}\leq\norm{f\br{s,\tilde{u}\br{s}}-f\br{t,\tilde{u}\br{s}}}_{\E}+\norm{f\br{t,\tilde{u}\br{s}}-f\br{t,\tilde{u}\br{t}}}_{\E},
	\end{equation}
	for $s,t\in\sbrt{a,b}$, it is rather easy to verify that $f\br{\cdot,\tilde{u}\br{\cdot}}\in\Crd{\sbrt{a,b}}{\E}$. Thus, applying \theoremRef{thm_03} on $\tilde{u}$, the claim follows.
\end{proof}

\begin{remark}
	\theoremRef{thm_05} might be more natural than \theoremRef{thm_04}, as one often works with rd-continuity in the context of time scales. Yet, \theoremRef{thm_04} is important, as it coincides with the well-known result from the theory of classical differential equations in Banach spaces.
\end{remark}

\begin{remark}
	Note that \assumptionRef{thm_05_assum_02} of \theoremRef{thm_05} cannot be replaced by the weaker assumption of $f\br{t,\cdot}$ being continuous on $\ball{\E}{u_0}{\beta}$, for every $t\in\sbrt{a,b}$. Otherwise, with such a weaker assumption, it might not necessarily hold that $f\br{\cdot,u\br{\cdot}}\in\Crd{\sbrt{a,b}}{\E}$, for $u\in\C{\sbrt{a,b}}{\E}$. For a counterexample, see \exampleRef{ex_04}.
\end{remark}

\begin{example}\label{ex_04}
	Consider $\E=\R$, $\T\defeq\R$ and $f\st\R^2\to\R$ defined as $f\br{0,0}\defeq0$, and $f\br{t,x}\defeq\frac{tx}{t^2+x^2}$, for $\br{0,0}\neq\br{t,x}\in\R^2$. Obviously, $f\br{t,\cdot}$ is continuous on $\R$, for every $t\in\R$, and $f\br{\cdot,x}$ is continuous on $\R$, for every $x\in\R$. However, for $u$ defined as $u\br{t}\defeq{t}$, for $t\in\R$ (i.e. for $u$ being the identity function), the composition $f\br{\cdot,u\br{\cdot}}$ is not continuous at $0$, as one can easily verify, despite the fact that $u$ is continuous on $\R$.
\end{example}

\subsection{Measures of noncompactness}
In this subsection we introduce the axiomatic theory of measures of noncompactness, which is thoroughly described in Chapter $3$ in \cite{Banas_Goebel}. Also, we prove some further auxiliary results which are necessary for our main theorem.

Firstly, we define $\fM_{\E}$ as the family of all the non-empty and bounded subsets of $\E$. Similarly, we define $\fN_{\E}$ as the family of all the non-empty and relatively compact subsets of $\E$ (i.e. $\fN_{\E}$ is the family of all the non-empty subsets of $\E$ with compact closure).

Let $X\subseteq\cbr{u\st\sbrt{a,b}\to\E}$ be a non-empty family of functions. For $t\in\sbrt{a,b}$, we define $X\br{t}\defeq\cbr{u\br{t}\st{u}\in{X}}\subseteq\E$. Moreover (provided that the integrals below are well-defined), for $t\in\sbrt{a,b}$, we define
\begin{equation*}
	\DeltaAint{a}{t}{X\br{s}\Deltas}\defeq\cbr{\DeltaAint{a}{t}{u\br{s}\Deltas}\st{u}\in{X}}\subseteq\E.
\end{equation*}
In addition, let $f\st\sbrt{a,b}\times\E\to\E$ be arbitrary. In the same manner as above, for $t\in\sbrt{a,b}$, we define $f\br{t,X\br{t}}\defeq\cbr{f\br{t,u\br{t}}\st{u}\in{X}}\subseteq\E$. Analogously, we define $f\br{\cdot,X\br{\cdot}}\defeq\cbr{f\br{\cdot,u\br{\cdot}}\st{u}\in{X}}\subseteq\cbr{u\st\sbrt{a,b}\to\E}$. Furthermore (provided that the integrals below are well-defined), for $t\in\sbrt{a,b}$, we define
\begin{equation*}
	\DeltaAint{a}{t}{f\br{s,X\br{s}}\Deltas}\defeq\cbr{\DeltaAint{a}{t}{f\br{s,u\br{s}}\Deltas}\st{u}\in{X}}\subseteq\E.
\end{equation*}

Finally, for $X,Y\subseteq\E$ and $\lambda_1,\lambda_2\in\R$, the linear combination $\lambda_1{X}+\lambda_2{Y}$ is defined as
\begin{equation*}
	\lambda_1{X}+\lambda_2{Y}\defeq\cbr{\lambda_1{x}+\lambda_2{y}\st{x}\in{X},y\in{Y}}\subseteq\E.
\end{equation*}

\begin{definition}[\normalfont{Measure of noncompactness}]\label{def_05}
	The mapping $\mu\st\fM_{\E}\to\lsrbr{0,\infty}$ is called a \textit{measure of noncompactness} in the space $\E$, if the following conditions hold
	\begin{enumerate}[(i)]
		\item\label{def_05_ax_01} The family $\ker\br{\mu}\defeq\cbr{X\in\fM_{\E}\st\mu\br{X}=0}$ is non-empty and $\ker\br{\mu}\subseteq\fN_{\E}$.
		\item\label{def_05_ax_02} If $X\subseteq{Y}$, then $\mu\br{X}\leq\mu\br{Y}$, for all $X,Y\in\fM_{\E}$.
		\item\label{def_05_ax_03} $\mu\br{\cl{X}}=\mu\br{X}$, for all $X\in\fM_{\E}$.
		\item\label{def_05_ax_04} $\mu\br{\conv{X}}=\mu\br{X}$, for all $X\in\fM_{\E}$.
		\item\label{def_05_ax_05} $\mu\br{\lambda{X}+\br{1-\lambda}Y}\leq\lambda\cdot\mu\br{X}+\br{1-\lambda}\cdot\mu\br{Y}$, for all $\lambda\in\sbr{0,1}$, for all $X,Y\in\fM_{\E}$.
		\item\label{def_05_ax_06} If $\cbr{X_k}_{k=1}^{\infty}\subseteq\fM_{\E}$ is a sequence of closed sets such that $X_{k+1}\subseteq{X}_k$, for all $k=1,2,\dots$, and if $\lim_{k\to\infty}\mu\br{X_k}=0$, then $\bigcap_{k=1}^{\infty}X_{k}\neq\emptyset$.
	\end{enumerate}
\end{definition}

\begin{remark}\label{rem_13}
	Let $\mu$ be a measure of noncompactness in the space $\E$. Let $\cbr{X_k}_{k=1}^{\infty}\subseteq\fM_{\E}$ be a sequence of closed sets such that $X_{k+1}\subseteq{X}_k$, for all $k=1,2,\dots$, and that $\lim_{k\to\infty}\mu\br{X_k}=0$. Then, it is rather easy to show that $X_{\infty}\defeq\bigcap_{k=1}^{\infty}X_k\in\fN_{\E}$ (i.e. $X_{\infty}$ is non-empty and relatively compact in $\E$). Indeed, \axiomRef{def_05_ax_06} of \definitionRef{def_05} readily implies that $X_{\infty}\neq\emptyset$. Moreover, since $X_{\infty}\subseteq{X}_k$, for every $k=1,2,\dots$, using the fact that $\lim_{k\to\infty}\mu\br{X_k}=0$ and \axiomRef{def_05_ax_02} of \definitionRef{def_05}, it can be easily shown that $\mu\br{X_{\infty}}=0$. Thus, according to \axiomRef{def_05_ax_01} of \definitionRef{def_05}, $X_{\infty}\in\ker\br{\mu}\subseteq\fN_{\E}$.
\end{remark}

\begin{remark}\label{rem_14}
	Let $\mu$ be a measure of noncompactness in the space $\E$. Let $\cbr{u_k}_{k=1}^{\infty}\subseteq\E$ be a sequence such that
	\begin{equation*}
		\lim_{k\to\infty}\mu\br{\cbr{u_k,u_{k+1},u_{k+2},\dots}}=0.
	\end{equation*}
	With the use of \axiomRef{def_05_ax_03} and \axiomRef{def_05_ax_06} of \definitionRef{def_05}, it is rather easy to show that $\bigcap_{k=1}^{\infty}\cl{\cbr{u_k,u_{k+1},u_{k+2},\dots}}\neq\emptyset$, which means that $\cbr{u_k}_{k=1}^{\infty}\subseteq\E$ has at least one cluster point (i.e. $\cbr{u_k}_{k=1}^{\infty}\subseteq\E$ contains a convergent subsequence).
\end{remark}

\begin{definition}\label{def_06}
	We say that a measure of noncompactness $\mu$ in the space $\E$ is \textit{sublinear}, if
	\begin{enumerate}[(i)]\addtocounter{enumi}{6}
		\item $\mu\br{\lambda{X}}=\abs{\lambda}\cdot\mu\br{X}$, for all $\lambda\in\R$, for all $X\in\fM_{\E}$.
		\item $\mu\br{X+Y}\leq\mu\br{X}+\mu\br{Y}$, for all $X,Y\in\fM_{\E}$.
	\end{enumerate}
\end{definition}

\begin{remark}
	In the theory of time scales, $\mu$ usually denotes the so called \textit{graininess} function. However, in this paper we use $\mu$ to denote a measure of noncompactness, as it is also often used in this context in the literature. Moreover, as we will work with measures of noncompactness in different spaces, a measure of noncompactness $\mu$ in the space $\E$ will be denoted by $\mu_{\E}$.
\end{remark}

\begin{lemma}\label{lm_03}
	Let $\mu_{\E}$ be a sublinear measure of noncompactness and $X\subseteq\cbr{u\st\sbrt{a,b}\to\E}$ be a non-empty family of equibounded functions. Then for all $s,t\in\sbrt{a,b}$ it holds that
	\begin{equation}\label{eq_02_20}
		\abs{\mu_{\E}\br{X\br{s}}-\mu_{\E}\br{X\br{t}}}\leq{K}_{\mu_{\E}}\cdot\sup_{u\in{X}}\norm{u\br{s}-u\br{t}}_{\E},
	\end{equation}
	where $K_{\mu_{\E}}\defeq\mu_{\E}\br{\ball{\E}{0_{\E}}{1}}\geq0$ is a non-negative constant.
\end{lemma}
\begin{proof}
	Firstly, note that since the family $X$ is non-empty and equibounded, $\mu_{\E}\br{X\br{t}}$ is well-defined (i.e. $X\br{t}\in\fM_{\E}$), for all $t\in\sbrt{a,b}$.
	
	According to Chapter $3$ in \cite{Banas_Goebel}, for every sublinear measure of noncompactness $\mu_{\E}$, and for all $Y,Z\in\fM_{\E}$, it holds that
	\begin{equation}\label{eq_02_21}
		\abs{\mu_{\E}\br{Y}-\mu_{\E}\br{Z}}\leq\mu_{\E}\br{\ball{\E}{0_{\E}}{1}}\cdot\max\cbr{d\br{Y,Z},d\br{Z,Y}},
	\end{equation}
	where $d\br{Y,Z}$ is defined as
	\begin{equation}\label{eq_02_22}
		d\br{Y,Z}\defeq\inf\cbr{r>0\st{Y}\subseteq\bigcup_{z\in{Z}}\openball{\E}{z}{r}}.
	\end{equation}
	
	Let $s,t\in\sbrt{a,b}$ be arbitrary. Define
	\begin{equation*}
		\nu\defeq\sup_{u\in{X}}\norm{u\br{s}-u\br{t}}_{\E}.
	\end{equation*}
	Obviously, it holds that
	\begin{equation*}
		\br{\forall{r}>0}\br{\forall{u}\in{X}}\st\br{u\br{s}\in\openball{\E}{u\br{t}}{\nu+r}},
	\end{equation*}
	which yields that
	\begin{equation}\label{eq_02_23}
		\br{\forall{r}>0}\st\br{X\br{s}\subseteq\bigcup_{u\in{X}}\openball{\E}{u\br{t}}{\nu+r}}.
	\end{equation}
	In the view of \eqref{eq_02_22}, we immediately obtain from \eqref{eq_02_23} that $d\br{X\br{s},X\br{t}}\leq\nu$. Similarly, it holds that $d\br{X\br{t},X\br{s}}\leq\nu$. Keeping in mind these estimates, \eqref{eq_02_20} follows from \eqref{eq_02_21}.
\end{proof}

\begin{lemma}\label{lm_04}
	Let $\mu_{\E}$ be a sublinear measure of noncompactness. Let $X\subseteq\Crd{\sbrt{a,b}}{\E}$ be a non-empty, uniformly rd-equicontinuous and equibounded family on $\sbrt{a,b}$. Then for all $t_0\in\sbrt{a,b}$ it holds that
	\begin{equation}\label{eq_02_24}
		\mu_{\E}\br{\DeltaAint{a}{t_0}{X\br{s}\Deltas}}\leq\Deltaint{a}{t_0}{\mu_{\E}\br{X\br{s}}\Deltas}.
	\end{equation}
\end{lemma}
\begin{proof}
	Firstly, note that the integrals on the LHS of \eqref{eq_02_24} are well-defined, since $X\subseteq\Crd{\sbrt{a,b}}{\E}$. Moreover, with the use of the uniform rd-equicontinuity of $X$, together with \lemmaRef{lm_03}, it can be easily shown that $\mu_{\E}\br{X\br{\cdot}}$ is rd-continuous on $\sbrt{a,b}$. Thus, the integral on the RHS of \eqref{eq_02_24} is also well-defined. Finally, note that the equiboundedness of $X$ implies the equiboundedness of the family $\DeltaAint{a}{t_0}{X\br{s}\Deltas}$, which means that $\mu_{\E}\br{\DeltaAint{a}{t_0}{X\br{s}\Deltas}}$ is well-defined as well, for every $t_0\in\sbrt{a,b}$.
	
	Let $t_0\in\sbrt{a,b}$ and $\eps>0$ be arbitrary. Consider the corresponding partition $\cbr{\xi_i}_{i=1}^m$ from \lemmaRef{lm_02} such that the following (simplified) conclusion of \eqref{eq_02_13} holds
	\begin{equation}\label{eq_02_25}
		\br{\forall{u}\in{X}}\br{\forall{j}=1,\dots,m-1}\br{\forall{s}\in\lsrbrt{\xi_j,\xi_{j+1}}}\st\br{\norm{u\br{s}-u\br{\xi_j}}_{\E}<\eps}.
	\end{equation}
	Let $u\in{X}$ be arbitrary. It holds that
	\begin{equation}\label{eq_02_26}
		\begin{aligned}
			&\norm{\DeltaAint{a}{t_0}{u\br{s}\Deltas}-\sum_{j=1}^{m-1}\br{\Deltaint{\xi_j}{\xi_{j+1}}{\Deltas}\cdot{u}\br{\xi_j}}}_{\E}=\norm{\sum_{j=1}^{m-1}\DeltaAint{\xi_j}{\xi_{j+1}}{\br{u\br{s}-u\br{\xi_j}}\Deltas}}_{\E}\\
			&\qquad\leqText{\eqref{eq_02_25}\,\&\,\eqref{eq_02_03}}\sum_{j=1}^{m-1}\Deltaint{\xi_j}{\xi_{j+1}}{\eps\Deltas}=\br{t_0-a}\cdot\eps.
		\end{aligned}
	\end{equation}
	Since $u\in{X}$ was chosen as arbitrary, by adding a \enquote{clever zero}, it follows that
	\begin{equation}\label{eq_02_27}
		\begin{aligned}
			&\DeltaAint{a}{t_0}{X\br{s}\Deltas}\\
			&\qquad\subseteq\cbr{\br{\DeltaAint{a}{t_0}{u\br{s}\Deltas}-\sum_{j=1}^{m-1}\br{\Deltaint{\xi_j}{\xi_{j+1}}{\Deltas}\cdot{u}\br{\xi_j}}}\st{u}\in{X}}+\\
			&\qquad\qquad\qquad+\cbr{\br{\sum_{j=1}^{m-1}\br{\Deltaint{\xi_j}{\xi_{j+1}}{\Deltas}\cdot{u}\br{\xi_j}}}\st{u}\in{X}}\\
			&\qquad\subseteqText{\eqref{eq_02_26}}\br{t_0-a}\cdot\eps\cdot\ball{\E}{0_{\E}}{1}+\cbr{\br{\sum_{j=1}^{m-1}\br{\Deltaint{\xi_j}{\xi_{j+1}}{\Deltas}\cdot{u}\br{\xi_j}}}\st{u}\in{X}}\\
			&\qquad\subseteq\br{t_0-a}\cdot\eps\cdot\ball{\E}{0_{\E}}{1}+\sum_{j=1}^{m-1}\br{\Deltaint{\xi_j}{\xi_{j+1}}{\Deltas}\cdot\cbr{u\br{\xi_j}\st{u}\in{X}}},
		\end{aligned}
	\end{equation}
	where the last inclusion follows from the fact that for all $\lambda_1,\lambda_2\in\R$, and for all $\tau_1,\tau_2\in\sbrt{a,b}$, it holds that
	\begin{equation*}
		\cbr{\br{\lambda_1\cdot{u}\br{\tau_1}+\lambda_2\cdot{u}\br{\tau_2}}\st{u}\in{X}}\subseteq\cbr{\br{\lambda_1\cdot{u}\br{\tau_1}+\lambda_2\cdot{v}\br{\tau_2}}\st{u},v\in{X}}\subseteq\E.
	\end{equation*}
	Hence, by \definitionRef{def_06} and \axiomRef{def_05_ax_02} of \definitionRef{def_05}, \eqref{eq_02_27} implies that
	\begin{equation}\label{eq_02_28}
		\mu_{\E}\br{\DeltaAint{a}{t_0}{X\br{s}\Deltas}}\leq\br{t_0-a}\cdot\eps\cdot\mu_{\E}\br{\ball{\E}{0_{\E}}{1}}+\sum_{j=1}^{m-1}\br{\Deltaint{\xi_j}{\xi_{j+1}}{\Deltas}\cdot\mu_{\E}\br{X\br{\xi_j}}}.
	\end{equation}
	
	Note that the LHS of \eqref{eq_02_28} does not depend on $\eps$. Moreover, in the limit $\eps\to0^+$, the first term on the RHS of \eqref{eq_02_28} vanishes. Thus, in order to prove \eqref{eq_02_24}, it suffices to show that (recall that the partition $\cbr{\xi_i}_{i=1}^m$ depends on $\eps$)
	\begin{equation}\label{eq_02_29}
		\sum_{j=1}^{m-1}\br{\Deltaint{\xi_j}{\xi_{j+1}}{\Deltas}\cdot\mu_{\E}\br{X\br{\xi_j}}}\xrightarrow{\eps\to{0^+}}\Deltaint{a}{t_0}{\mu_{\E}\br{X\br{s}}\Deltas}.
	\end{equation}
	
	From \lemmaRef{lm_03}, it follows that for all $j=1,\dots,m-1$, and for all $s\in\lsrbrt{\xi_j,\xi_{j+1}}$, it holds that
	\begin{equation}\label{eq_02_30}
		\abs{\mu_{\E}\br{X\br{s}}-\mu_{\E}\br{X\br{\xi_j}}}\leqText{\lmRef{lm_03}}K_{\mu_{\E}}\cdot\sup_{u\in{X}}\norm{u\br{s}-u\br{\xi_j}}_{\E}\leqText{\eqref{eq_02_25}}K_{\mu_{\E}}\cdot\eps.
	\end{equation}
	Finally, using \eqref{eq_02_30}, by the very same arguments as in \eqref{eq_02_26}, we easily obtain \eqref{eq_02_29}.
\end{proof}

\begin{theorem}\label{thm_06}
	Let $\mu_{\E}$ be a measure of noncompactness in $\E$. Then the mapping $\mu_{\C{\sbrt{a,b}}{\E}}\st\fM_{\C{\sbrt{a,b}}{\E}}\to\lsrbr{0,\infty}$ defined as
	\begin{equation}\label{eq_02_31}
		\mu_{\C{\sbrt{a,b}}{\E}}\br{X}\defeq\lim_{\eps\to0^+}\sup_{u\in{X}}\,\sup_{s,t\in\sbrt{a,b}\,\st\abs{s-t}\leq\eps}\norm{u\br{s}-u\br{t}}_{\E}+\sup_{t\in\sbrt{a,b}}\mu_{\E}\br{X\br{t}},
	\end{equation}
	for $X\in\fM_{\C{\sbrt{a,b}}{\E}}$, is a measure of noncompactness in the space $\C{\sbrt{a,b}}{\E}$.
\end{theorem}
\begin{proof}
	Since $\mu_{\E}$ is a measure of noncompactness in $\E$, by \axiomRef{def_05_ax_01} of \definitionRef{def_05}, there exists $X\in\fM_{\E}$ such that $\mu_{\E}\br{X}=0$. Moreover, by \axiomRef{def_05_ax_02} of \definitionRef{def_05}, $\mu_{\E}\br{\cbr{x}}=0$, for every $x\in{X}$. Thus, apparently $\ker\br{\mu_{\C{\sbrt{a,b}}{\E}}}\neq\emptyset$, since for every $x\in{X}$ and the constant function $u\equiv{x}$ on $\sbrt{a,b}$, we have that $\mu_{\C{\sbrt{a,b}}{\E}}\br{\cbr{u}}=0$.

	Let $X\in\fM_{\C{\sbrt{a,b}}{\E}}$ be arbitrary such that $\mu_{\C{\sbrt{a,b}}{\E}}\br{X}=0$. In particular, \eqref{eq_02_31} implies that
	\begin{equation*}
		\lim_{\eps\to0^+}\sup_{u\in{X}}\,\sup_{s,t\in\sbrt{a,b}\,\st\abs{s-t}\leq\eps}\norm{u\br{s}-u\br{t}}_{\E}=0,
	\end{equation*}
	which readily implies that the family of functions $X$ is uniformly equicontinuous on $\sbrt{a,b}$. Furthermore, \eqref{eq_02_31} implies also that
	\begin{equation*}
		\sup_{t\in\sbrt{a,b}}\mu_{\E}\br{X\br{t}}=0,
	\end{equation*}
	which, according to \axiomRef{def_05_ax_01} of \definitionRef{def_05}, implies that $X\br{t}$ is relatively compact in $\E$, for every $t\in\sbrt{a,b}$. Thus, according to \theoremRef{thm_02}, the set $X$ is relatively compact in $\C{\sbrt{a,b}}{\E}$. Hence, $\ker\br{\mu_{\C{\sbrt{a,b}}{\E}}}\subseteq\fN_{\C{\sbrt{a,b}}{\E}}$, which means that \axiomRef{def_05_ax_01} of \definitionRef{def_05} is satisfied for $\mu_{\C{\sbrt{a,b}}{\E}}$.

	Using the fact that $\mu_{\E}$ is a measure of noncompactness in $\E$, it is rather easy to verify that \axiomRef{def_05_ax_02} of \definitionRef{def_05} is satisfied for $\mu_{\C{\sbrt{a,b}}{\E}}$.

	As for \axiomRef{def_05_ax_03} of \definitionRef{def_05}, let $X\in\fM_{\C{\sbrt{a,b}}{\E}}$ be arbitrary. Firstly, note that it can be easily shown that $\br{\cl{X}}\br{t}=\cl{X\br{t}}$, for all $t\in\sbrt{a,b}$. Indeed, let $t\in\sbrt{a,b}$ be arbitrary. Then, $x\in\br{\cl{X}}\br{t}$ means precisely that there exists $\cbr{u_k}_{k=1}^{\infty}\subseteq{X}$ such that $\lim_{k\to\infty}u_k\br{t}=x$. On the other hand, $x\in\cl{X\br{t}}$ means precisely that there exists $\cbr{x_k}_{k=1}^{\infty}\subseteq{X}\br{t}$ such that $\lim_{k\to\infty}x_k=x$. However, considering that $X\br{t}=\cbr{u\br{t}\st{u}\in{X}}$, it follows that $x_k\in{X}\br{t}$ if and only if there exists $u_k\in{X}$ such that $u_k\br{t}=x_k$, for $k=1,2,\dots$ Thus, $x\in\cl{X\br{t}}$ means precisely that there exists $\cbr{u_k}_{k=1}^{\infty}\subseteq{X}$ such that $\lim_{k\to\infty}u_k\br{t}=x$, as well. Hence, the claim $\br{\cl{X}}\br{t}=\cl{X\br{t}}$ is true. Using this fact, together with the fact that $\mu_{\E}$ is a measure of noncompactness in $\E$, it is obvious that in order to prove that
	\begin{equation*}
		\mu_{\C{\sbrt{a,b}}{\E}}\br{\cl{X}}=\mu_{\C{\sbrt{a,b}}{\E}}\br{X},
	\end{equation*}
	it suffices to show that
	\begin{equation}\label{eq_02_32}
		\lim_{\eps\to0^+}\sup_{u\in\cl{X}}\,\sup_{s,t\in\sbrt{a,b}\,\st\abs{s-t}\leq\eps}\norm{u\br{s}-u\br{t}}_{\E}\leq\lim_{\eps\to0^+}\sup_{u\in{X}}\,\sup_{s,t\in\sbrt{a,b}\,\st\abs{s-t}\leq\eps}\norm{u\br{s}-u\br{t}}_{\E},
	\end{equation}
	as the other direction from \eqref{eq_02_32} is trivial. Let $u_{\infty}\in\cl{X}$ be arbitrary. Then there exists $\cbr{u_k}_{k=1}^{\infty}\subseteq{X}$ such that $u_k\rightrightarrows{u}_{\infty}$. Let $s,t\in\sbrt{a,b}$ be arbitrary. By the continuity of $\norm{\cdot}_{\E}$, it immediately follows that
	\begin{equation}\label{eq_02_33}
		\norm{u_{\infty}\br{s}-u_{\infty}\br{t}}_{\E}=\lim_{k\to\infty}\norm{u_k\br{s}-u_k\br{t}}_{\E}\leqText{\cbr{u_k}_{k=1}^{\infty}\subseteq{X}}\sup_{u\in{X}}\norm{u\br{s}-u\br{t}}_{\E}.
	\end{equation}
	Since $u_{\infty}\in\cl{X}$ and $s,t\in\sbrt{a,b}$ were chosen as arbitrary, \eqref{eq_02_33} readily implies that \eqref{eq_02_32} holds. Hence, \axiomRef{def_05_ax_03} of \definitionRef{def_05} is satisfied for $\mu_{\C{\sbrt{a,b}}{\E}}$.

	As for \axiomRef{def_05_ax_04} of \definitionRef{def_05}, let $X\in\fM_{\C{\sbrt{a,b}}{\E}}$ be arbitrary. Firstly, note that by similar arguments as in the previous paragraph, it can be easily shown that $\br{\conv{X}}\br{t}=\conv{X\br{t}}$, for all $t\in\sbrt{a,b}$. Again, using this fact, together with the fact that $\mu_{\E}$ is a measure of noncompactness in $\E$, it is obvious that in order to prove that
	\begin{equation*}
		\mu_{\C{\sbrt{a,b}}{\E}}\br{\conv{X}}=\mu_{\C{\sbrt{a,b}}{\E}}\br{X},
	\end{equation*}
	it suffices to show that
	\begin{equation}\label{eq_02_34}
		\lim_{\eps\to0^+}\sup_{u\in\conv{X}}\,\sup_{s,t\in\sbrt{a,b}\,\st\abs{s-t}\leq\eps}\norm{u\br{s}-u\br{t}}_{\E}\leq\lim_{\eps\to0^+}\sup_{u\in{X}}\,\sup_{s,t\in\sbrt{a,b}\,\st\abs{s-t}\leq\eps}\norm{u\br{s}-u\br{t}}_{\E},
	\end{equation}
	as the other direction from \eqref{eq_02_34} is again trivial. Let $u_0\in\conv{X}$ be arbitrary. Then there exists $\lambda_1,\dots,\lambda_N>0$ and $\cbr{u_k}_{k=1}^{N}\subseteq{X}$ such that $\sum_{k=1}^N\lambda_k=1$ and $\sum_{k=1}^N\lambda_k{u}_k=u_0$. Let $s,t\in\sbrt{a,b}$ be arbitrary. Using the triangle inequality, it readily follows that
	\begin{equation}\label{eq_02_35}
		\begin{aligned}
			&\norm{u_0\br{s}-u_0\br{t}}_{\E}=\norm{\sum_{k=1}^N\lambda_k\cdot\br{u_k\br{s}-u_k\br{t}}}_{\E}\leq\sum_{k=1}^N\lambda_k\cdot\norm{u_k\br{s}-u_k\br{t}}_{\E}\\
			&\qquad\leqText{\cbr{u_k}_{k=1}^{N}\subseteq{X}}\br{\sum_{k=1}^N\lambda_k}\cdot\sup_{u\in{X}}\norm{u\br{s}-u\br{t}}_{\E}=\sup_{u\in{X}}\norm{u\br{s}-u\br{t}}_{\E}.
		\end{aligned}
	\end{equation}
	Since $u_0\in\conv{X}$ and $s,t\in\sbrt{a,b}$ were chosen as arbitrary, \eqref{eq_02_35} readily implies that \eqref{eq_02_34} holds. Hence, \axiomRef{def_05_ax_04} of \definitionRef{def_05} is satisfied for $\mu_{\C{\sbrt{a,b}}{\E}}$.

	Next, note that using the triangle inequality, together with the fact that $\mu_{\E}$ is a measure of noncompactness in $\E$, it is rather easy to verify that \axiomRef{def_05_ax_05} of \definitionRef{def_05} is satisfied for $\mu_{\C{\sbrt{a,b}}{\E}}$.

	Finally, as for \axiomRef{def_05_ax_06} of \definitionRef{def_05}, let $\cbr{X_k}_{k=1}^{\infty}\subseteq\fM_{\C{\sbrt{a,b}}{\E}}$ be a sequence of closed sets such that $X_{k+1}\subseteq{X}_k$, for all $k=1,2,\dots$, and that $\lim_{k\to\infty}\mu_{\C{\sbrt{a,b}}{\E}}\br{X_k}=0$. In particular, note that for every $t\in\sbrt{a,b}$, \eqref{eq_02_31} implies that
	\begin{equation}\label{eq_02_36}
		\lim_{k\to\infty}\mu_{\E}\br{X_k\br{t}}=0.
	\end{equation}

	Now, let $\cbr{u_k}_{k=1}^{\infty}\subseteq\C{\sbrt{a,b}}{\E}$ be arbitrary such that $u_k\in{X}_k$, for all $k=1,2,\dots$ In particular, since ${X}_1\supseteq{X}_2\supseteq\dots$, it holds that $\cbr{u_k\br{t},u_{k+1}\br{t},u_{k+2}\br{t},\dots}\subseteq{X}_k\br{t}$, for all $t\in\sbrt{a,b}$, for all $k=1,2,\dots$ Thus, utilising \axiomRef{def_05_ax_02} of \definitionRef{def_05}, for every $t\in\sbrt{a,b}$, it readily follows that
	\begin{equation*}
		\mu_{\E}\br{\cbr{u_k\br{t},u_{k+1}\br{t},u_{k+2}\br{t},\dots}}\leqText{\defRef{def_05}}\mu_{\E}\br{X_k\br{t}}\xrightarrow{k\to\infty}0,
	\end{equation*}
	where the last limit follows from \eqref{eq_02_36}. Hence, with respect to \remarkRef{rem_14}, $\cbr{u_k\br{t}}_{k=1}^{\infty}\subseteq\E$ contains a convergent subsequence (or equivalently, $\cbr{u_k}_{k=1}^{\infty}\subseteq\C{\sbrt{a,b}}{\E}$ contains a subsequence which is pointwise convergent at the point $t$), for all $t\in\sbrt{a,b}$. Let $\cbr{t_k}_{k=1}^{\infty}\subseteq\sbrt{a,b}$ be a sequence of real numbers dense in $\sbrt{a,b}$, and let $\cbr{u_k^1}_{k=1}^{\infty}$ be a subsequence of $\cbr{u_k}_{k=1}^{\infty}$ which is pointwise convergent at the point $t_1$. Inductively, for $l=1,2,\dots$, let $\cbr{u_k^{l+1}}_{k=1}^{\infty}\subseteq\cbr{u_k^l}_{k=1}^{\infty}$ be a subsequence of $\cbr{u_k^l}_{k=1}^{\infty}$ which is pointwise convergent at the point $t_{l+1}$. In particular, note that the sequence $\cbr{u_k^k}_{k=1}^{\infty}$ is pointwise convergent at all the points $\cbr{t_k}_{k=1}^{\infty}$. Further on, without loss of generality, we can assume that $\cbr{u_k}_{k=1}^{\infty}\subseteq\C{\sbrt{a,b}}{\E}$ is such a sequence, which is pointwise convergent at all the points $\cbr{t_k}_{k=1}^{\infty}$. Finally, define $u_{\infty}\st\cbr{t_k}_{k=1}^{\infty}\to\E$ as
	\begin{equation}\label{eq_02_37}
		u_{\infty}\br{t_j}\defeq\lim_{k\to\infty}u_k\br{t_j},\quad{j}=1,2,\dots,
	\end{equation}
	which according to the previous discussion is well-defined.

	Next, for $X\in\fM_{\C{\sbrt{a,b}}{\E}}$ and $\eps>0$, define (compare with the first term on the RHS of \eqref{eq_02_31})
	\begin{equation}\label{eq_02_38}
		\omega\br{X,\eps}\defeq\sup_{u\in{X}}\,\sup_{s,t\in\sbrt{a,b}\,\st\abs{s-t}\leq\eps}\norm{u\br{s}-u\br{t}}_{\E}.
	\end{equation}
	Let $k=1,2,\dots$ and $i,j=1,2,\dots$ be arbitrary. Then it follows that
	\begin{equation}\label{eq_02_39}
		\norm{u_k\br{t_i}-u_k\br{t_j}}_{\E}\leqText{\eqref{eq_02_38}}\omega\br{X_k,\abs{t_i-t_j}}.
	\end{equation}
	In the limit $k\to\infty$, and with respect to \eqref{eq_02_37}, \eqref{eq_02_39} yields that
	\begin{equation}\label{eq_02_40}
		\norm{u_{\infty}\br{t_i}-u_{\infty}\br{t_j}}_{\E}\leq\lim_{k\to\infty}\omega\br{X_k,\abs{t_i-t_j}}.
	\end{equation}
	Note that the RHS of \eqref{eq_02_40} is well-defined and finite, as the sequence $\cbr{\omega\br{X_k,\abs{t_i-t_j}}}_{k=1}^{\infty}$ is non-increasing (due to the fact that $X_1\supseteq{X}_2\supseteq\dots$) and bounded (due to the boundedness of $X_1$). Moreover, since $\lim_{k\to\infty}\mu_{\C{\sbrt{a,b}}{\E}}\br{X_k}=0$, it can be easily shown that
	\begin{equation}\label{eq_02_41}
		\lim_{\eps\to0^+}\lim_{k\to\infty}\omega\br{X_k,\eps}=0.
	\end{equation}
	In particular, combining \eqref{eq_02_40} and \eqref{eq_02_41}, it readily follows that $u_{\infty}\st\cbr{t_k}_{k=1}^{\infty}\to\E$ is uniformly continuous on the set $\cbr{t_k}_{k=1}^{\infty}$. Finally, since $\cbr{t_k}_{k=1}^{\infty}$ is dense in $\sbrt{a,b}$, $u_{\infty}$ can be extended (uniquely) to the whole interval $\sbrt{a,b}$ such that (compare with \eqref{eq_02_40})
	\begin{equation}\label{eq_02_42}
		\norm{u_{\infty}\br{s}-u_{\infty}\br{t}}_{\E}\leq\lim_{k\to\infty}\omega\br{X_k,\abs{s-t}},\quad{s},t\in\sbrt{a,b}.
	\end{equation}
	Again, note that \eqref{eq_02_42} and \eqref{eq_02_41} imply that such an extended $u_{\infty}\st\sbrt{a,b}\to\E$ is uniformly continuous on $\sbrt{a,b}$ (i.e. $u\in\C{\sbrt{a,b}}{\E}$).

	Now, note that it can be easily verified that in order to show that $\bigcap_{k=1}^{\infty}X_k\neq\emptyset$, it suffices to show that the extended $u_{\infty}$ is the uniform limit of $\cbr{u_k}_{k=1}^{\infty}$ (recall that $\cbr{u_k}_{k=1}^{\infty}$ is such that $u_k\in{X}_k$, for all $k=1,2,\dots$, and that $X_1\supseteq{X}_2\supseteq\dots$ are closed). Let $\eps>0$ be arbitrary. According to \eqref{eq_02_41}, there exists $\delta>0$ such that for all $0<h<\delta$, it holds that
	\begin{equation}\label{eq_02_43}
		\lim_{k\to\infty}\omega\br{X_k,h}<\eps.
	\end{equation}
	Moreover, let $t\in\sbrt{a,b}$ be arbitrary. Since $\cbr{t_k}_{k=1}^{\infty}$ is dense in $\sbrt{a,b}$, there exists $j\in\N$ such that $\abs{t-t_j}<\delta$. Then, for all $k=1,2,\dots$, it follows that
	\begin{equation}\label{eq_02_44}
		\begin{aligned}
			&\norm{u_k\br{t}-u_{\infty}\br{t}}_{\E}\leq\norm{u_k\br{t}-u_k\br{t_j}}_{\E}+\norm{u_k\br{t_j}-u_{\infty}\br{t_j}}_{\E}+\norm{u_{\infty}\br{t_j}-u_{\infty}\br{t}}_{\E}\\
			&\qquad\leqText{\eqref{eq_02_38}\,\&\,\eqref{eq_02_42}}\omega\br{X_k,\abs{t-t_j}}+\norm{u_k\br{t_j}-u_{\infty}\br{t_j}}_{\E}+\lim_{l\to\infty}\omega\br{X_l,\abs{t-t_j}}.
		\end{aligned}
	\end{equation}
	Thus, with respect to \eqref{eq_02_37}, applying the $\limsup_{k\to\infty}$ operator on both the sides of \eqref{eq_02_44}, we readily obtain that
	\begin{equation}\label{eq_02_45}
		\limsup_{k\to\infty}\norm{u_k\br{t}-u_{\infty}\br{t}}_{\E}\leq2\cdot\lim_{k\to\infty}\omega\br{X_k,\abs{t-t_j}}\lText{\abs{t-t_j}<\delta\,\&\,\eqref{eq_02_43}}2\eps.
	\end{equation}
	Finally, since $\eps>0$ and $t\in\sbrt{a,b}$ were chosen as arbitrary, and since $t\in\sbrt{a,b}$ is independent of the choice of $\delta>0$, we readily obtain from \eqref{eq_02_45} that $u_k\rightrightarrows{u}_{\infty}$ on $\sbrt{a,b}$. Thus, as discussed above at the beginning of this paragraph, this implies that $\bigcap_{k=1}^{\infty}X_k\neq\emptyset$ (i.e. \axiomRef{def_05_ax_06} of \definitionRef{def_05} is satisfied as well).
\end{proof}

\subsection{Kamke $\Delta$-function}\label{sec_02_05}
In this subsection we define the notion of a Kamke $\Delta$-function, which plays the role of a comparison condition in our main theorem. Note that this concept is new and was introduced by us.

\begin{definition}[\normalfont{Kamke $\Delta$-function}]\label{def_07}
	We say that $w\st\sbrt{a,b}\times\lsrbr{0,\infty}\to\lsrbr{0,\infty}$ is a \textit{Kamke $\Delta$-function} on $\sbrt{a,b}$, if the following conditions hold
	\begin{enumerate}[(i)]
		\item\label{def_07_ax_01} $w\br{\cdot,x}$ is rd-continuous on $\sbrt{a,b}$, for all $x\in\lsrbr{0,\infty}$.
		\item\label{def_07_ax_02} The family $\cbr{w\br{t,\cdot}}_{t\in\sbrt{a,b}}$ is uniformly equicontinuous on $\sbr{0,x}$, for all $x\in\br{0,\infty}$.
		\item\label{def_07_ax_03} $w\br{t,0}=0$, for all $t\in\sbrt{a,b}$.
		\item\label{def_07_ax_04} $u\equiv0$ is the unique non-negative continuous function satisfying, for all $t\in\sbrt{a,b}$, the following integral inequality
		\begin{equation}\label{eq_02_46}
			u\br{t}\leq\Deltaint{a}{t}{w\br{s,u\br{s}}\Deltas},\quad{t}\in\sbrt{a,b},
		\end{equation}
		for which $u\br{\cdot}$ is $\Delta$-differentiable at $a$ and $u^{\Delta}\br{a}=0$.		
	\end{enumerate}
\end{definition}

\begin{remark}
	Note that the integral in \eqref{eq_02_46} is well-defined. Indeed, let $u\in\C{\sbrt{a,b}}{\R}$ be an arbitrary non-negative function. Using the triangle inequality $\abs{w\br{s,u\br{s}}-w\br{t,u\br{t}}}\leq\abs{w\br{s,u\br{s}}-w\br{t,u\br{s}}}+\abs{w\br{t,u\br{s}}-w\br{t,u\br{t}}}$, for $s,t\in\sbrt{a,b}$, together with the continuity of $u$, and \axiomRef{def_07_ax_01} and \axiomRef{def_07_ax_02} of \definitionRef{def_07}, it can be easily shown that $w\br{\cdot,u\br{\cdot}}\in\Crd{\sbrt{a,b}}{\R}$. Thus, according to \remarkRef{rem_03}, the integral on the RHS of \eqref{eq_02_46} is well-defined.
\end{remark}

\begin{remark}\label{rem_17}
	Consider \axiomRef{def_07_ax_04} of \definitionRef{def_07}. Let $u\in\C{\sbrt{a,b}}{\R}$ be a non-negative function satisfying \eqref{eq_02_46}, for all $t\in\sbrt{a,b}$. From \eqref{eq_02_46}, it immediately follows that $u\br{a}=0$. Moreover
	\begin{enumerate}[a)]
		\item If $a$ is right-scattered, then $u\br{\cdot}$ is trivially $\Delta$-differentiable at $a$ (since it is continuous at the right-scattered point $a$). Thus, by \eqref{eq_02_01}, showing that $u\br{\cdot}$ is $\Delta$-differentiable at $a$ and that $u^{\Delta}\br{a}=0$, is equivalent to showing that $u\br{\sigma\br{a}}=0$.
		\item If $a$ is right-dense, then by \eqref{eq_02_02}, showing that $u\br{\cdot}$ is $\Delta$-differentiable at $a$ and that $u^{\Delta}\br{a}=0$, is equivalent to showing that the limit
		\begin{equation*}
			\lim_{t\to{a},t\in\T}\frac{u\br{t}}{t-a},
		\end{equation*}
		exists and is equal to $0$.
	\end{enumerate}
\end{remark}

\begin{remark}\label{rem_18}
	Due to \axiomRef{def_07_ax_03} of \definitionRef{def_07}, $u\equiv0$ always satisfies \eqref{eq_02_46}, for all $t\in\sbrt{a,b}$. Moreover, $u\equiv0$ trivially satisfies the additional requirement regarding $\Delta$-differentiability at $a$. The question is whether it is the only non-negative continuous function satisfying both of these. In particular, if $u\equiv0$ is the only non-negative solution of \eqref{eq_02_46} (regardless of whether or not the additional requirement regarding $\Delta$-differentiability at $a$ holds), \axiomRef{def_07_ax_04} of \definitionRef{def_07} is satisfied.
\end{remark}

\begin{remark}\label{rem_19}
	It can be shown that if $\tilde{w}$ is a Kamke $\Delta$-function on $\sbrt{a,b}$, then for every $c\in\sbrt{a,b}$, $w\defeq\tilde{w}\arrowvert_{\sbrt{a,\sigma\br{c}}\times\lsrbr{0,\infty}}$ is a Kamke $\Delta$-function on $\sbrt{a,\sigma\br{c}}$. Indeed, such a $w$ apparently satisfies \axiomRef{def_07_ax_02}, \axiomRef{def_07_ax_03}, and according to \remarkRef{rem_07}, also \axiomRef{def_07_ax_01} of \definitionRef{def_07}. Now, assume that $w$ does not satisfy \axiomRef{def_07_ax_04} of \definitionRef{def_07} (i.e. assume that there exists a non-negative $0\not\equiv{u}_0\in\C{\sbrt{a,\sigma\br{c}}}{\R}$ satisfying \eqref{eq_02_46}, for all $t\in\sbrt{a,\sigma\br{c}}$, such that $u_0$ is $\Delta$-differentiable at $a$ and $u_0^{\Delta}\br{a}=0$). Let $\tilde{u}_0$ be a continuous extension of $u_0$ to $\sbrt{a,b}$ defined as $\tilde{u}_0\br{t}\defeq{u}_0\br{t}$, for $t\in\sbrt{a,\sigma\br{c}}$, and $\tilde{u}_0\br{t}\defeq{u}_0\br{\sigma\br{c}}$, for $t\in\lbrsbrt{\sigma\br{c},b}$. By its construction, $\tilde{u}_0\not\equiv0$ is continuous and non-negative on $\sbrt{a,b}$, it satisfies \eqref{eq_02_46}, for all $t\in\sbrt{a,\sigma\br{c}}$, it is $\Delta$-differentiable at $a$ and $\tilde{u}_0^{\Delta}\br{a}=0$. Moreover, using the additivity of Cauchy $\Delta$-integral, together with the fact that $\tilde{w}$ is non-negative (by its definition), it is easy to see that $\tilde{u}_0\br{t}=u_0\br{\sigma\br{c}}\leq\Deltaint{a}{\sigma\br{c}}{w\br{s,u_0\br{s}}\Deltas}=\Deltaint{a}{\sigma\br{c}}{\tilde{w}\br{s,\tilde{u}_0\br{s}}\Deltas}\leq\Deltaint{a}{t}{\tilde{w}\br{s,\tilde{u}_0\br{s}}\Deltas}$, for all $t\in\lbrsbrt{\sigma\br{c},b}$. Thus, we obtain a contradiction with \axiomRef{def_07_ax_04} of \definitionRef{def_07} and the fact that $\tilde{w}$ is a Kamke $\Delta$-function on $\sbrt{a,b}$, which means that $\tilde{w}\arrowvert_{\sbrt{a,\sigma\br{c}}\times\lsrbr{0,\infty}}$ is a Kamke $\Delta$-function on $\sbrt{a,\sigma\br{c}}$.
\end{remark}

\subsection{Schauder fixed point theorem}
\begin{theorem}[\normalfont{Schauder fixed point theorem}]\label{thm_07}
	Let $X\subseteq\E$ be a non-empty, bounded, closed and convex set. Let $\fF\st{X}\to{X}$ be a continuous mapping such that $\fF\br{X}$ is relatively compact in $\E$. Then there exists a fixed point of $\fF$ in $X$.
\end{theorem}

\section{Main Result}\label{sec_03}
Now that we have all the preliminary results we need, we are ready to prove our main existence theorem, which is an extension of Theorem $13.3.1$ in \cite{Banas_Goebel} to the case of an arbitrary time scale. Also, note that the main ideas of the proof presented below were motivated by the proof of the above-mentioned theorem from \cite{Banas_Goebel}, and the proof of Theorem $14$ in \cite{Oberta} (which is a similar result for the case of fractional differential equations in Banach spaces).

\begin{theorem}\label{thm_08}
	Let $\E$ be a real Banach space, $\sbrt{a,b}$ a time scale interval and $u_0\in\E$. Let $f\st\sbrt{a,b}\times\ball{\E}{u_0}{\beta}\to\E$, for some $\beta>0$. Let $\mu_{\E}$ be a sublinear measure of noncompactness in $\E$, and $w$ be a Kamke $\Delta$-function on $\sbrt{a,b}$. Assume that
	\begin{enumerate}[(i)]
		\item\label{thm_08_assum_01} The family $\cbr{f\br{\cdot,x}}_{x\in\ball{\E}{u_0}{\beta}}$ is uniformly rd-equicontinuous on $\sbrt{a,b}$.
		\item\label{thm_08_assum_02} The family $\cbr{f\br{t,\cdot}}_{t\in\sbrt{a,b}}$ is uniformly equicontinuous on $\ball{\E}{u_0}{\beta}$.
		\item\label{thm_08_assum_03} $\norm{f\br{t,x}}_{\E}\leq{M}$, for all $\br{t,x}\in\sbrt{a,b}\times\ball{\E}{u_0}{\beta}$, for some $M>0$.
		\item\label{thm_08_assum_04} $\mu_{\E}\br{f\br{t,X}}\leq{w}\br{t,\mu_{\E}\br{X}}$, for all $t\in\sbrt{a,b}$, for all $\emptyset\neq{X}\subseteq\ball{\E}{u_0}{\beta}$.
		\item\label{thm_08_assum_05} $\mu_{\E}\br{\cbr{u_0}}=0$.
	\end{enumerate}
	Then the IVP
	\begin{subequations}\label{eq_03_01}
		\begin{align}
			u^{\Delta}\br{t}&=f\br{t,u\br{t}},\quad{t}\in\sbrtk{a,b},\\
			u\br{a}&=u_0,\label{eq_03_01_02}
		\end{align}
	\end{subequations}
	has a local solution on $\sbrt{a,\sigma\br{\bast}}$, where $\bast\in\sbrt{a,b}$ is defined as
	\begin{equation}\label{eq_03_02}
		\bast\defeq\sup\cbr{s\in\T\st\sigma\br{s}\leq\min\cbr{b,a+\frac{\beta}{M}}}.
	\end{equation}
\end{theorem}
\begin{proof}
	Firstly we provide a brief outline of the main steps of the proof
	\begin{enumerate}[1)]
		\item We create a nested sequence $\cbr{X_k}_{k=0}^{\infty}$ of non-empty, bounded, closed, convex and uniformly equicontinuous subsets of $\C{\sbrt{a,\sigma\br{\bast}}}{\E}$.
		\item We show that the sequence $\cbr{\mu_{\E}\br{X_k\br{\cdot}}}_{k=0}^{\infty}$ (in the space $\C{\sbrt{a,\sigma\br{\bast}}}{\R}$) is uniformly equicontinuous and equibounded on $\sbrt{a,\sigma\br{\bast}}$. Using the Arzelà-Ascoli theorem, we show that this sequence converges uniformly to some $v_{\infty}\in\C{\sbrt{a,\sigma\br{\bast}}}{\R}$.
		\item We verify that $v_{\infty}$ satisfies given integral inequality from \axiomRef{def_07_ax_04} of \definitionRef{def_07}, as well as the additional requirement regarding $\Delta$-differentiability at $a$. Thus, we show that $v_{\infty}\equiv0$ on $\sbrt{a,\sigma\br{\bast}}$ (i.e. $\mu_{\E}\br{X_k\br{\cdot}}\rightrightarrows0$ on $\sbrt{a,\sigma\br{\bast}}$).
		\item We show, using the axioms of \definitionRef{def_05} (for the measure of noncompactness in $\C{\sbrt{a,\sigma\br{\bast}}}{\E}$, as defined in \theoremRef{thm_06}), that the set $X_{\infty}\defeq\bigcap_{k=0}^{\infty}X_k$ is non-empty and relatively compact.
		\item We show, using the Schauder fixed point theorem, that the corresponding Volterra-type integral equation \eqref{eq_02_18} has a solution in $X_{\infty}$, which according to \theoremRef{thm_05} is a local solution of \eqref{eq_03_01}.
	\end{enumerate}
	Next, we provide details for these steps. Firstly, note that \eqref{eq_03_02} implies that
	\begin{equation}\label{eq_03_03}
		M\cdot\br{\sigma\br{\bast}-a}\leq\beta.
	\end{equation}
	
	Define $X_0$ as
	\begin{equation}\label{eq_03_04}
		\begin{aligned}
			X_0\defeq\{u\st&\,u\in\C{\sbrt{a,\sigma\br{\bast}}}{\E};\;\norm{u\br{\cdot}-u_0}_{\infty}\leq\beta;\;u\br{a}=u_0;\\
			&\norm{u\br{s}-u\br{t}}_{\E}\leq{M}\cdot\abs{s-t},\forall{s},t\in\sbrt{a,\sigma\br{\bast}}\}.
		\end{aligned}
	\end{equation}
	It is rather easy to verify that the family $X_0$ is non-empty, bounded, closed, convex and uniformly equicontinuous on $\sbrt{a,\sigma\br{\bast}}$. Moreover, according to \remarkRef{rem_08}, \assumptionRef{thm_08_assum_01} implies that the family $\cbr{f\br{\cdot,x}}_{x\in\ball{\E}{u_0}{\beta}}$ is uniformly rd-equicontinuous on $\sbrt{a,\sigma\br{\bast}}$. Using this fact and \assumptionRef{thm_08_assum_02} (whilst keeping in mind the uniform (rd-)equicontinuity in both of these), together with the uniform equicontinuity of $X_0$ on $\sbrt{a,\sigma\br{\bast}}$, and the triangle inequality \eqref{eq_02_19}, one can verify that the family $\cbr{f\br{\cdot,u\br{\cdot}}}_{u\in{X}_0}$ is uniformly rd-equicontinuous on $\sbrt{a,\sigma\br{\bast}}$.
	
	Define the operator $\fF\st{X}_0\to\C{\sbrt{a,\sigma\br{\bast}}}{\E}$ as the RHS of \eqref{eq_02_18}, i.e.
	\begin{equation}\label{eq_03_05}
		\br{\fF\br{u}}\br{t}\defeq{u}_0+\DeltaAint{a}{t}{f\br{s,u\br{s}}\Deltas},\quad{t}\in\sbrt{a,\sigma\br{\bast}},
	\end{equation}
	for $u\in{X}_0$.
	
	As discussed above, $f\br{\cdot,u\br{\cdot}}$ is rd-continuous on $\sbrt{a,\sigma\br{\bast}}$, for every $u\in{X}_0$. This implies that $\DeltaAint{a}{\cdot}{f\br{s,u\br{s}}\Deltas}$ is well-defined and continuous on $\sbrt{a,\sigma\br{\bast}}$. Thus, the definition of $\fF$ is correct (i.e. $\fF$ maps $X_0$ into $\C{\sbrt{a,\sigma\br{\bast}}}{\E}$). Moreover, according to \theoremRef{thm_05}, it is enough to show that $\fF$ has a fixed point.
	
	Now, we verify that $\fF\br{X_0}\subseteq{X}_0$. Let $u\in{X}_0$ be arbitrary. Apparently, \eqref{eq_03_05} implies that $\br{\fF\br{u}}\br{a}=u_0$, and as discussed above, $\br{\fF\br{u}}\br{\cdot}$ is continuous on $\sbrt{a,\sigma\br{\bast}}$. Moreover, it holds that
	\begin{equation*}
		\begin{aligned}
			&\norm{\fF\br{u}-u_0}_{\infty}=\sup_{t\in\sbrt{a,\sigma\br{\bast}}}\norm{\br{\fF\br{u}}\br{t}-u_0}_{\E}\eqText{\eqref{eq_03_05}}\sup_{t\in\sbrt{a,\sigma\br{\bast}}}\norm{\DeltaAint{a}{t}{f\br{s,u\br{s}}\Deltas}}_{\E}\\
			&\qquad\leqText{\assumRef{thm_08_assum_03}\,\&\,\eqref{eq_02_03}}\sup_{t\in\sbrt{a,\sigma\br{\bast}}}\Deltaint{a}{t}{M\Deltas}=M\cdot\br{\sigma\br{\bast}-a}\leqText{\eqref{eq_03_03}}\beta.
		\end{aligned}
	\end{equation*}
	Finally, let $s,t\in\sbrt{a,\sigma\br{\bast}}$ be arbitrary such that $s\leq{t}$. It holds that
	\begin{equation*}
		\begin{aligned}
			&\norm{\br{\fF\br{u}}\br{s}-\br{\fF\br{u}}\br{t}}_{\E}\eqText{\eqref{eq_03_05}}\norm{\DeltaAint{s}{t}{f\br{\tau,u\br{\tau}}\Deltatau}}_{\E}\\
			&\qquad\leqText{\assumRef{thm_08_assum_03}\,\&\,\eqref{eq_02_03}}\Deltaint{s}{t}{M\Deltatau}=M\cdot\abs{s-t}.
		\end{aligned}
	\end{equation*}
	Hence, $\fF\br{X_0}\subseteq{X}_0$.
	
	Next, define $\cbr{X_k}_{k=1}^{\infty}$ as
	\begin{equation}\label{eq_03_06}
		X_{k+1}\defeq\cl{\conv{\fF\br{X_k}}},\quad{k}=0,1,\dots
	\end{equation}
	Since $X_0$ is closed and convex, and $\fF\br{X_0}\subseteq{X}_0$, \eqref{eq_03_06} implies that $\fF\br{X_0}\subseteq{X}_1\subseteq{X}_0$. In particular, this implies that $\fF\br{X_1}\subseteq{X}_1$. Considering that $X_1$ is closed and convex (by its definition), \eqref{eq_03_06} implies that $\fF\br{X_1}\subseteq{X_2}\subseteq{X}_1$. Inductively, we obtain that
	\begin{equation}\label{eq_03_07}
		X_{k+1}\subseteq{X}_k,\quad{k}=0,1,\dots
	\end{equation}
	Note that in the view of the uniform equicontinuity of $X_0$, \eqref{eq_03_07} implies the uniform equicontinuity of $X_k$ on $\sbrt{a,\sigma\br{\bast}}$, for every $k=0,1,\dots$
	
	Define $v_k\st\sbrt{a,\sigma\br{\bast}}\to\lsrbr{0,\infty}$ as
	\begin{equation}\label{eq_03_08}
		v_k\br{\cdot}\defeq\mu_{\E}\br{X_k\br{\cdot}},\quad{k}=0,1,\dots
	\end{equation}
	Let $k=0,1,\dots$ and $s,t\in\sbrt{a,\sigma\br{\bast}}$ be arbitrary. Using \lemmaRef{lm_03}, we readily obtain that
	\begin{equation*}
		\abs{v_k\br{s}-v_k\br{t}}\leqText{\lmRef{lm_03}}{K}_{\mu_{\E}}\cdot\sup_{u\in{X}_k}\norm{u\br{s}-u\br{t}}_{\E}\leqText{X_k\subseteq{X}_0}{K}_{\mu_{\E}}\cdot\sup_{u\in{X}_0}\norm{u\br{s}-u\br{t}}_{\E},
	\end{equation*}
	which in the view of the independence of the RHS on $k$, and the uniform equicontinuity of $X_0$ (in the space $\C{\sbrt{a,\sigma\br{\bast}}}{\E}$), implies the uniform equicontinuity of $\cbr{v_k}_{k=0}^{\infty}$ (in the space $\C{\sbrt{a,\sigma\br{\bast}}}{\R}$). Moreover, from \eqref{eq_03_07} and \axiomRef{def_05_ax_02} of \definitionRef{def_05}, it follows that
	\begin{equation}\label{eq_03_09}
		0\leq{v}_{k+1}\br{t}\leq{v}_k\br{t}.
	\end{equation}
	Furthermore, since $v_0$ is continuous, it is bounded on $\sbrt{a,\sigma\br{\bast}}$. Thus, \eqref{eq_03_09} implies the equiboundedness of $\cbr{v_k}_{k=0}^{\infty}$ on $\sbrt{a,\sigma\br{\bast}}$, which means that the set $\cbr{v_k\br{t}}_{k=0}^{\infty}\subseteq\R$ is relatively compact in $\R$, for every $t\in\sbrt{a,\sigma\br{\bast}}$ (recall that  $J\subseteq\R$ is relatively compact if and only if it is bounded). Hence, according to \theoremRef{thm_02}, $\cbr{v_k}_{k=0}^{\infty}$ is relatively compact in $\C{\sbrt{a,\sigma\br{\bast}}}{\R}$, which (by one of the equivalent characteristics of relative compactness) implies that it contains a uniformly convergent subsequence. By \eqref{eq_03_09}, it follows that the whole sequence is uniformly convergent, i.e. there exists $v_{\infty}\in\C{\sbrt{a,\sigma\br{\bast}}}{\R}$ such that
	\begin{equation}\label{eq_03_10}
		v_k\rightrightarrows{v}_{\infty}\geqText{\eqref{eq_03_09}}0.
	\end{equation}
	
	Now, using \axiomRef{def_07_ax_04} of \definitionRef{def_07}, we show that $v_{\infty}\equiv0$. Let $k=0,1,\dots$ and $t\in\sbrt{a,\sigma\br{\bast}}$ be arbitrary. Firstly, note that it is rather easy to verify that (it was actually shown as a part of the proof of \theoremRef{thm_06})
	\begin{equation}\label{eq_03_11}
		\br{\cl{\conv{\fF\br{X_k}}}}\br{t}=\cl{\conv{\br{\fF\br{X_k}}\br{t}}}.
	\end{equation}
	Utilising \axiomRef{def_05_ax_03} and \axiomRef{def_05_ax_04} of \definitionRef{def_05}, we readily obtain that
	\begin{equation}\label{eq_03_12}
		\begin{aligned}
			&v_{k+1}\br{t}\eqText{\eqref{eq_03_08}\,\&\,\eqref{eq_03_06}}\mu_{\E}\br{\br{\cl{\conv{\fF\br{X_k}}}}\br{t}}\eqText{\eqref{eq_03_11}\,\&\,\defRef{def_05}}\mu_{\E}\br{\br{\fF\br{X_k}}\br{t}}\\
			&\qquad\eqText{\eqref{eq_03_05}}\mu_{\E}\br{\cbr{u_0}+\DeltaAint{a}{t}{f\br{s,X_k\br{s}}\Deltas}}\\
			&\qquad\leqText{\defRef{def_06}\,\&\,\assumRef{thm_08_assum_05}}\mu_{\E}\br{\DeltaAint{a}{t}{f\br{s,X_k\br{s}}\Deltas}}\leqText{\lmRef{lm_04}}\Deltaint{a}{t}{\mu_{\E}\br{f\br{s,X_k\br{s}}}\Deltas}\\
			&\qquad\leqText{\assumRef{thm_08_assum_04}\,\&\,\eqref{eq_02_03}}\Deltaint{a}{t}{w\br{s,\mu_{\E}\br{X_k\br{s}}}\Deltas}\eqText{\eqref{eq_03_08}}\Deltaint{a}{t}{w\br{s,v_k\br{s}}\Deltas}.
		\end{aligned}
	\end{equation}
	Using the fact that $v_k\rightrightarrows{v}_{\infty}$ on $\sbrt{a,\sigma\br{\bast}}$, \axiomRef{def_07_ax_02} of \definitionRef{def_07}, and the equiboundedness of $\cbr{v_k}_{k=0}^{\infty}$, it is easy to verify that $w\br{\cdot,v_k\br{\cdot}}\rightrightarrows{w}\br{\cdot,v_{\infty}\br{\cdot}}$ on $\sbrt{a,\sigma\br{\bast}}$. Thus, it holds that
	\begin{equation}\label{eq_03_13}
		\begin{aligned}
			&\abs{\Deltaint{a}{t}{\br{w\br{s,v_k\br{s}}-w\br{s,v_{\infty}\br{s}}}\Deltas}}\\
			&\qquad\leqText{\eqref{eq_02_03}}\sup_{s\in\sbrt{a,\sigma\br{\bast}}}\abs{w\br{s,v_k\br{s}}-w\br{s,v_{\infty}\br{s}}}\cdot\Deltaint{a}{t}{\Deltas}\xrightarrow{k\to\infty}0,
		\end{aligned}
	\end{equation}
	where the last limit follows from the above-mentioned uniform convergence $w\br{\cdot,v_k\br{\cdot}}\rightrightarrows{w}\br{\cdot,v_{\infty}\br{\cdot}}$ on $\sbrt{a,\sigma\br{\bast}}$. Hence, taking the pointwise limit in \eqref{eq_03_12}, \eqref{eq_03_13} implies that $v_{\infty}$ satisfies the integral inequality from \eqref{eq_02_46}, i.e.
	\begin{equation}\label{eq_03_14}
		v_{\infty}\br{t}\leq\Deltaint{a}{t}{w\br{s,v_{\infty}\br{s}}\Deltas}.
	\end{equation}
	It remains to verify that $v_{\infty}$ satisfies also the additional requirement from \axiomRef{def_07_ax_04} of \definitionRef{def_07} (i.e. that $v_{\infty}\br{\cdot}$ is $\Delta$-differentiable at $a$ and that $v_{\infty}^{\Delta}\br{a}=0$).
	
	If $a$ is right-scattered, according to \remarkRef{rem_17} and \eqref{eq_03_09}, it is enough to verify that $v_1\br{\sigma\br{a}}=0$. Using the same arguments as in \eqref{eq_03_12}, it follows that
	\begin{equation*}
		\begin{aligned}
			&v_1\br{\sigma\br{a}}\leq\Deltaint{a}{\sigma\br{a}}{w\br{s,\mu_{\E}\br{X_0\br{s}}}\Deltas}\eqText{\eqref{eq_02_04}}\br{\sigma\br{a}-a}\cdot{w}\br{a,\mu_{\E}\br{X_0\br{a}}}\\
			&\qquad\eqText{\eqref{eq_03_04}}\br{\sigma\br{a}-a}\cdot{w}\br{a,\mu_{\E}\br{\cbr{u_0}}}\eqText{\assumRef{thm_08_assum_05}}\br{\sigma\br{a}-a}\cdot{w}\br{a,0}\eqText{\axRef{def_07_ax_03}\textrm{ of }\defRef{def_07}}0.
		\end{aligned}
	\end{equation*}
	
	Similarly as in the previous case, if $a$ is right-dense, according to \remarkRef{rem_17} and \eqref{eq_03_09}, it is enough to verify that $\lim_{t\to{a},t\in\T}\frac{v_1\br{t}}{t-a}$ exists and is equal to $0$. Define $y\br{\cdot}$ as
	\begin{equation}\label{eq_03_15}
		y\br{t}\defeq\br{t-a}\cdot{f}\br{a,u_0}+u_0,\quad{t}\in\sbrt{a,\sigma\br{\bast}}.
	\end{equation}
	Let $x\in\fF\br{{X}_0}$ be arbitrary (i.e. $x\br{\cdot}=\br{\fF\br{u}}\br{\cdot}$, for some $u\in{X}_0$). Firstly, note that it is easy to show that for all $X\subseteq\E$, it holds that $\sup_{z\in\cl{\conv{X}}}\norm{z}_{\E}=\sup_{z\in{X}}\norm{z}_{\E}$ (note that a similar relation was actually shown as a part of the proof of \theoremRef{thm_06}; see also \eqref{eq_02_33} and \eqref{eq_02_35}). Now, let $t\in\sbrt{a,\sigma\br{\bast}}$ be arbitrary. Keeping in mind the previous claim, it follows that
	\begin{equation}\label{eq_03_16}
		\begin{aligned}
			&\norm{y\br{t}-x\br{t}}_{\E}\eqText{\eqref{eq_03_05}\,\&\,\eqref{eq_03_15}}\norm{\br{t-a}\cdot{f}\br{a,u_0}-\DeltaAint{a}{t}{f\br{s,u\br{s}}\Deltas}}_{\E}\\
			&\qquad\leqText{\thmRef{thm_01}}\sup_{z\in\cl{\conv{\cbr{f\br{s,u\br{s}}\,\st{s}\in\sbrt{a,t}}}}}\norm{\br{t-a}\cdot{f}\br{a,u_0}-\br{t-a}\cdot{z}}_{\E}\\
			&\qquad=\br{t-a}\cdot\sup_{s\in\sbrt{a,t}}\norm{f\br{a,u_0}-f\br{s,u\br{s}}}_{\E}\leq\br{t-a}\cdot\gamma\br{t},
		\end{aligned}
	\end{equation}
	where $\gamma\br{\cdot}$ is defined as
	\begin{equation}\label{eq_03_17}
		\gamma\br{t}\defeq\sup_{v\in{X}_0}\sup_{s\in\sbrt{a,t}}\norm{f\br{a,u_0}-f\br{s,v\br{s}}}_{\E},\quad{t}\in\sbrt{a,\sigma\br{\bast}}.
	\end{equation}
	Notice that \eqref{eq_03_17}, and thus also the RHS of \eqref{eq_03_16}, is independent of $x\in\fF\br{{X}_0}$. Moreover, as \eqref{eq_03_16} holds for every $x\in\fF\br{{X}_0}$, using the sublinearity and continuity of $\norm{\cdot}_{\E}$, it is easy to show that \eqref{eq_03_16} holds for every $x\in\cl{\conv{\fF\br{X_0}}}=X_1$ as well. Hence, it follows that
	\begin{equation}\label{eq_03_18}
		X_1\br{t}\subseteq\ball{\E}{y\br{t}}{\br{t-a}\cdot\gamma\br{t}}=\cbr{y\br{t}}+\br{t-a}\cdot\gamma\br{t}\cdot\ball{\E}{0_{\E}}{1}.
	\end{equation}
	By \axiomRef{def_05_ax_02} of \definitionRef{def_05}, we readily obtain that
	\begin{equation}\label{eq_03_19}
		\begin{aligned}
			&v_1\br{t}\eqText{\eqref{eq_03_08}}\mu_{\E}\br{X_1\br{t}}\leqText{\eqref{eq_03_18}\,\&\,\defRef{def_05}\,\&\,\defRef{def_06}}\mu_{\E}\br{\cbr{y\br{t}}}+\br{t-a}\cdot\gamma\br{t}\cdot\mu_{\E}\br{\ball{\E}{0_{\E}}{1}}\\
			&\qquad=\br{t-a}\cdot\gamma\br{t}\cdot\mu_{\E}\br{\ball{\E}{0_{\E}}{1}},
		\end{aligned}
	\end{equation}
	where the last equality follows from
	\begin{equation*}
		\begin{aligned}
			&\mu_{\E}\br{\cbr{y\br{t}}}\leqText{\eqref{eq_03_15}\,\&\,\defRef{def_06}\,\&\,\assumRef{thm_08_assum_05}}\br{t-a}\cdot\mu_{\E}\br{\cbr{f\br{a,u_0}}}\\
			&\qquad\leqText{\assumRef{thm_08_assum_04}}\br{t-a}\cdot{w}\br{a,\mu_{\E}\br{\cbr{u_0}}}\eqText{\assumRef{thm_08_assum_05}}\br{t-a}\cdot{w}\br{a,0}\eqText{\axRef{def_07_ax_03}\textrm{ of }\defRef{def_07}}0.
		\end{aligned}
	\end{equation*}
	Thus, according to \eqref{eq_03_19}, in order to show that $\lim_{t\to{a},t\in\T}\frac{v_1\br{t}}{t-a}=0$, it suffices to show that $\lim_{t\to{a},t\in\T}\gamma\br{t}=0$. Since $v\br{a}=u_0$, for all $v\in{X}_0$, $a$ is right-dense, and $f\br{\cdot,X_0\br{\cdot}}$ is uniformly rd-equicontinuous on $\sbrt{a,\sigma\br{\bast}}$, \definitionRef{def_04} implies that
	\begin{equation}\label{eq_03_20}
		\br{\forall\eps>0}\br{\exists\delta>0}\br{\forall{v}\in{X}_0}\br{\forall{s}\in\lsrbr{a,a+\delta}\cap\T}\st\br{\norm{f\br{a,{u_0}}-f\br{s,v\br{s}}}_{\E}<\eps}.
	\end{equation}
	Keeping in mind \eqref{eq_03_17}, \eqref{eq_03_20} implies that $\lim_{t\to{a},t\in\T}\gamma\br{t}=0$.

	To conclude, we have shown that $v_{\infty}\br{\cdot}$ is $\Delta$-differentiable at $a$ and that $v_{\infty}^{\Delta}\br{a}=0$. Moreover, with respect to \remarkRef{rem_19}, since $v_{\infty}$ is non-negative, continuous and it satisfies \eqref{eq_03_14}, for all $t\in\sbrt{a,\sigma\br{\bast}}$, \axiomRef{def_07_ax_04} of \definitionRef{def_07} implies that $v_{\infty}\equiv0$ on $\sbrt{a,\sigma\br{\bast}}$.

	Recall that $\cbr{X_k}_{k=0}^{\infty}$ is a sequence of non-empty, bounded, closed, convex and uniformly equicontinuous subsets of $\C{\sbrt{a,\sigma\br{\bast}}}{\E}$. Moreover, it satisfies \eqref{eq_03_07}, and by \eqref{eq_03_10}, it also holds that $\mu_{\E}\br{X_k\br{\cdot}}\rightrightarrows0$ on $\sbrt{a,\sigma\br{\bast}}$ (recall that $v_{\infty}\equiv0$ on $\sbrt{a,\sigma\br{\bast}}$).

	Now, define $X_{\infty}$ as
	\begin{equation}\label{eq_03_21}
		X_{\infty}\defeq\bigcap_{k=0}^{\infty}{X}_k.
	\end{equation}
	Obviously, $X_{\infty}$ is bounded. Moreover, $X_{\infty}$ is closed and convex, since the intersection of an arbitrary number of closed and convex sets is closed and convex.
	
	Let $\mu_{\C{\sbrt{a,\sigma\br{\bast}}}{\E}}$ be the measure of noncompactness (in the space $\C{\sbrt{a,\sigma\br{\bast}}}{\E}$) from \theoremRef{thm_06}. Let $k=0,1,\dots$ be arbitrary. From \eqref{eq_02_31}, it follows that
	\begin{equation}\label{eq_03_22}
		\mu_{\C{\sbrt{a,\sigma\br{\bast}}}{\E}}\br{X_k}=\lim_{\eps\to0^+}\sup_{u\in{X}_k}\,\sup_{s,t\in\sbrt{a,\sigma\br{\bast}}\,\st\abs{s-t}\leq\eps}\norm{u\br{s}-u\br{t}}_{\E}+\sup_{t\in\sbrt{a,\sigma\br{\bast}}}\mu_{\E}\br{X_k\br{t}}.
	\end{equation}
	Since $X_k$ is uniformly equicontinuous on $\sbrt{a,\sigma\br{\bast}}$, the first term on the RHS of \eqref{eq_03_22} vanishes. Moreover, since $\mu_{\E}\br{X_k\br{\cdot}}\rightrightarrows0$ on $\sbrt{a,\sigma\br{\bast}}$, \eqref{eq_03_22} yields that
	\begin{equation*}
		\mu_{\C{\sbrt{a,\sigma\br{\bast}}}{\E}}\br{X_k}\xrightarrow{k\to\infty}0.
	\end{equation*}
	Thus, in the view of \remarkRef{rem_13}, $X_{\infty}$ is non-empty and relatively compact (in the space $\C{\sbrt{a,\sigma\br{\bast}}}{\E}$).

	Since $\fF\br{X_k}\subseteq{X}_k$ (see \eqref{eq_03_06} and \eqref{eq_03_07}), for all $k=0,1,\dots$, \eqref{eq_03_21} implies that $\fF\br{X_{\infty}}\subseteq{X}_{\infty}$. Furthermore, since $X_{\infty}$ is relatively compact, $\fF\br{X_{\infty}}$ is relatively compact as well.
	
	Next, we show that the mapping $\fF\st{X}_{\infty}\to{X}_{\infty}$ is continuous. Let $\cbr{u_k}_{k=1}^{\infty}\subseteq{X}_{\infty}$ be arbitrary such that $u_k\rightrightarrows{u}_{\infty}$ on $\sbrt{a,\sigma\br{\bast}}$, for some $u_{\infty}\in{X}_{\infty}$. By the Heine criterion, it is enough to show that $\fF\br{u_k}\rightrightarrows\fF\br{u_{\infty}}$ on $\sbrt{a,\sigma\br{\bast}}$. Let $\eps>0$ be arbitrary. Using \assumptionRef{thm_08_assum_02} and the fact that $u_k\rightrightarrows{u}_{\infty}$ on $\sbrt{a,\sigma\br{\bast}}$, it is easy to show that there exists $n_0\in\N$ such that
	\begin{equation}\label{eq_03_23}
		\br{\forall{k}\geq{n}_0}\br{\forall{t}\in\sbrt{a,\sigma\br{\bast}}}\st\br{\norm{f\br{t,u_k\br{t}}-f\br{t,u_{\infty}\br{t}}}_{\E}<\eps}.
	\end{equation}
	Let $k\geq{n}_0$ be arbitrary. It holds that
	\begin{equation*}
		\begin{aligned}
			&\norm{\fF\br{u_k}-\fF\br{u_{\infty}}}_{\infty}=\sup_{t\in\sbrt{a,\sigma\br{\bast}}}\norm{\br{\fF\br{u_k}}\br{t}-\br{\fF\br{u_{\infty}}}\br{t}}_{\E}\\
			&\qquad\eqText{\eqref{eq_03_05}}\sup_{t\in\sbrt{a,\sigma\br{\bast}}}\norm{\DeltaAint{a}{t}{\br{f\br{s,u_k\br{s}}-f\br{s,u_{\infty}\br{s}}}\Deltas}}_{\E}\leqText{\eqref{eq_03_23}\,\&\,\eqref{eq_02_03}}\eps\cdot\br{\sigma\br{\bast}-a}.
		\end{aligned}
	\end{equation*}
	Since $\eps>0$ was chosen as arbitrary, this immediately implies that $\fF\br{u_k}\rightrightarrows\fF\br{u_{\infty}}$ on $\sbrt{a,\sigma\br{\bast}}$.
	
	Finally, according to \theoremRef{thm_07}, $\fF$ has a fixed point in $X_{\infty}$, which by \theoremRef{thm_05} implies that \eqref{eq_03_01} has a local solution on $\sbrt{a,\sigma\br{\bast}}$.
\end{proof}

\section{Applications}\label{sec_04}
In this section we apply \theoremRef{thm_08} on a countable system of dynamic equations on time scales, which arises from semi-discretisation of parabolic partial dynamic equations. In particular, we prove that under certain assumptions (specified in \exampleRef{ex_05}), solution of such a countable system exists in the space $c_0$.

\subsection{Preliminaries}
Firstly, we remind some examples of Banach spaces, which will be used in this section. Also, we provide particular examples of a measure of noncompactness in $c_0$ and a Kamke $\Delta$-function, respectively

We denote by $\bx=\br{x_k}_{k=1}^{\infty}$ a real sequence. Note that the space $c_0\defeq\cbr{\bx\st\lim_{k\to\infty}x_k=0}$ is a Banach space when equipped with the norm $\norm{\bx}_{c_0}\defeq\sup_{k=1,2,\dots}\abs{x_k}$. Also, recall that the space $\BC{\lsrbr{0,\infty}}{\R}\defeq\cbr{u\st\lsrbr{0,\infty}\to\R\st{u}\text{ is bounded and continuous on }\lsrbr{0,\infty}}$ is a Banach space when equipped with the norm $\norm{u}_{\BC{\lsrbr{0,\infty}}{\R}}\defeq\sup_{x\in\lsrbr{0,\infty}}\abs{u\br{x}}$.

\begin{theorem}\label{thm_09}
	The mapping $\chi_{c_0}\st\fM_{c_0}\to\lsrbr{0,\infty}$ defined as
	\begin{equation}\label{eq_04_01}
		\chi_{c_0}\br{X}\defeq\lim_{k\to\infty}\sup_{\bx\in{X}}\sup_{j\geq{k}}\abs{x_j},
	\end{equation}
	for $X\in\fM_{c_0}$, is a sublinear measure of noncompactness (the so called Hausdorff measure of noncompactness) in the space $c_0$. Moreover, for all $\bx\in{c}_0$ it holds that
	\begin{equation}\label{eq_04_02}
		\chi_{c_0}\br{\cbr{\bx}}=0.
	\end{equation}
\end{theorem}
\begin{proof}
	The theorem follows from Theorem $5.18$, Lemma $5.9$ and Lemma $5.10$ in \cite{Banas_Mursaleen}.
\end{proof}

\begin{theorem}\label{thm_10}
	Let $q\in\Crd{\sbrt{a,b}}{\R}$ be non-negative on $\sbrt{a,b}$. Then
	\begin{equation*}
		w\br{t,x}\defeq{q}\br{t}\cdot{x},\quad{t}\in\sbrt{a,b},x\in\lsrbr{0,\infty},
	\end{equation*}
	is a Kamke $\Delta$-function on $\sbrt{a,b}$.
\end{theorem}
\begin{proof}
	\axiomRef{def_07_ax_01} and \axiomRef{def_07_ax_03} of \definitionRef{def_07} are obviously satisfied. Moreover, since $q$ is bounded on $\sbrt{a,b}$ (as it is rd-continuous on $\sbrt{a,b}$), it is easy to verify that \axiomRef{def_07_ax_02} of \definitionRef{def_07} holds. Finally, according to Theorem $6.4$ in \cite{Bohner_Dynamic}, for every $u\in\C{\sbrt{a,b}}{\R}$ satisfying, for all $t\in\sbrt{a,b}$, the following integral inequality
	\begin{equation*}
		u\br{t}\leq\Deltaint{a}{t}{q\br{s}\cdot{u}\br{s}\Deltas},
	\end{equation*}
	it holds that $u\br{t}\leq0$, for all $t\in\sbrt{a,b}$. Thus, $u\equiv0$ is the unique non-negative continuous solution of \eqref{eq_02_46} on $\sbrt{a,b}$, which according to \remarkRef{rem_18} means that \axiomRef{def_07_ax_04} of \definitionRef{def_07} holds as well.
\end{proof}

\subsection{Dynamic equations on time scales in the space $c_0$}
Firstly, we rewrite the IVP \eqref{eq_03_01} in the form of a dynamic equation in a Banach sequence space. Consider the following IVP
\begin{subequations}\label{eq_04_03}
	\begin{align}
		\bu^{\Delta}\br{t}&=\bff\br{t,\bu\br{t}},\quad{t}\in\sbrtk{a,b},\\
		\bu\br{a}&=\bvarphi,
	\end{align}
\end{subequations}
where $\bu\br{\cdot}\defeq\br{u_k\br{\cdot}}_{k=1}^{\infty}$, $\mathbf{f}\br{\cdot,\bx}\defeq\br{f_k\br{\cdot,\bx}}_{k=1}^{\infty}$ and $\bvarphi\defeq\br{\varphi_k}_{k=1}^{\infty}$.

\begin{theorem}\label{thm_11}
	Consider the IVP \eqref{eq_04_03} in the Banach space $\E=c_0$. Let $\bvarphi\in{c}_0$ and $\bff\st\sbrt{a,b}\times\ball{c_0}{\bvarphi}{\beta}\to{c}_0$, for some $\beta>0$. Let $p_j\in\Crd{\sbrt{a,b}}{\R}$ be non-negative, for $j=1,2,\dots$ Let $\br{n_j}_{j=1}^{\infty}\subseteq\N$ be a non-decreasing sequence such that $\lim_{j\to\infty}n_j=\infty$. Assume that
	\begin{enumerate}[(i)]
		\item\label{thm_11_assum_01} The family $\cbr{\bff\br{\cdot,\bx}}_{\bx\in\ball{c_0}{\bvarphi}{\beta}}$ is uniformly rd-equicontinuous on $\sbrt{a,b}$.
		\item\label{thm_11_assum_02} The family $\cbr{\bff\br{t,\cdot}}_{t\in\sbrt{a,b}}$ is uniformly equicontinuous on $\ball{c_0}{\bvarphi}{\beta}$.
		\item\label{thm_11_assum_03} $\abs{f_j\br{t,\bx}}\leq{p}_j\br{t}+Q\sup_{i\geq{n}_j}\abs{x_i}$, for all $\br{t,\bx}\in\sbrt{a,b}\times\ball{c_0}{\bvarphi}{\beta}$, for all $j=1,2,\dots$, for some $Q>0$.
		\item\label{thm_11_assum_04} $\lim_{j\to\infty}p_j\br{t}=0$, for all $t\in\sbrt{a,b}$.
		\item\label{thm_11_assum_05} $p_j\br{t}\leq{P}$, for all $t\in\sbrt{a,b}$, for all $j=1,2,\dots$, for some $P>0$.
	\end{enumerate}
	Then the IVP \eqref{eq_04_03} has a local solution on $\sbrt{a,\sigma\br{\bast}}$, where $\bast\in\sbrt{a,b}$ is defined as in \eqref{eq_03_02} and $M$ is given by
	\begin{equation}\label{eq_04_04}
		M\defeq{P}+Q\cdot\br{\norm{\bvarphi}_{c_0}+\beta}.
	\end{equation}
\end{theorem}
\begin{proof}
	Obviously, \assumptionRef{thm_08_assum_01} and \assumptionRef{thm_08_assum_02} of \theoremRef{thm_08} are satisfied. Moreover, \assumptionRef{thm_08_assum_03} of \theoremRef{thm_08} is a direct consequence of \eqref{eq_04_04}, \assumptionRef{thm_11_assum_03} and \assumptionRef{thm_11_assum_05} of \theoremRef{thm_11}, and the fact that $\bx\in\ball{c_0}{\bvarphi}{\beta}$.
	
	Consider the Hausdorff measure of noncompactness in the space $c_0$ from \theoremRef{thm_09}. \assumptionRef{thm_08_assum_05} of \theoremRef{thm_08} is a direct consequence of \eqref{eq_04_02}.
	
	Finally, let $w$ be a Kamke $\Delta$-function on $\sbrt{a,b}$ defined as
	\begin{equation}\label{eq_04_05}
		w\br{t,x}\defeq{Q}\cdot{x},\quad{t}\in\sbrt{a,b},x\in\lsrbr{0,\infty}.
	\end{equation}
	Note that according to \theoremRef{thm_10}, such a $w$ is indeed a Kamke $\Delta$-function on $\sbrt{a,b}$. Let $t\in\sbrt{a,b}$ and $\emptyset\neq{X}\subseteq\ball{c_0}{\bvarphi}{\beta}$ be arbitrary. It holds that
	\begin{equation*}
		\begin{aligned}
			&\chi_{c_0}\br{\bff\br{t,X}}\eqText{\eqref{eq_04_01}}\lim_{k\to\infty}\sup_{\bx\in{X}}\sup_{j\geq{k}}\abs{f_j\br{t,\bx}}\leqText{\assumRef{thm_11_assum_03}}\lim_{k\to\infty}\sup_{\bx\in{X}}\sup_{j\geq{k}}\br{p_j\br{t}+Q\cdot\sup_{i\geq{n}_j}\abs{x_i}}\\
			&\qquad\leq\lim_{k\to\infty}\sup_{j\geq{k}}p_j\br{t}+Q\cdot\lim_{k\to\infty}\sup_{\bx\in{X}}\sup_{j\geq{k}}\sup_{i\geq{n}_j}\abs{x_i}\eqText{\assumRef{thm_11_assum_04}}Q\cdot\lim_{k\to\infty}\sup_{\bx\in{X}}\sup_{j\geq{k}}\sup_{i\geq{n}_j}\abs{x_i}\\
			&\qquad\eqText{\lim_{j\to\infty}n_j=\infty}Q\cdot\lim_{k\to\infty}\sup_{\bx\in{X}}\sup_{j\geq{k}}\abs{x_j}\eqText{\eqref{eq_04_01}}Q\cdot\chi_{c_0}\br{X}\eqText{\eqref{eq_04_05}}w\br{t,\chi_{c_0}\br{X}}.
		\end{aligned}
	\end{equation*}
	Thus, \assumptionRef{thm_08_assum_04} of \theoremRef{thm_08} is satisfied.
	
	Hence, according to \theoremRef{thm_08}, \eqref{eq_04_03} has a local solution on $\sbrt{a,\sigma\br{\bast}}$.
\end{proof}

\begin{remark}
	Note that instead of the non-negative functions $p_j\in\Crd{\sbrt{a,b}}{\R}$, for $j=1,2,\dots$ in \theoremRef{thm_11}, it suffices to assume a single non-negative abstract function $\bp=\br{p_j}_{j=1}^{\infty}\in\Crd{\sbrt{a,b}}{c_0}$ satisfying \assumptionRef{thm_11_assum_03} of \theoremRef{thm_11}. In such a case, \assumptionRef{thm_11_assum_04} and \assumptionRef{thm_11_assum_05} of \theoremRef{thm_11} would be satisfied automatically.
\end{remark}

\subsection{Semi-discretisation of parabolic partial dynamic equations on time scales}\label{sec_04_03}
Let $u\st\T\times\R\to\R$ and $\br{t_0,x_0}\in\Tk\times\R$. We define the partial $\Delta$-derivative $\pDelta{u}{t}$ of $u$ with respect to the first variable at the point $\br{t_0,x_0}$ as $\pDelta{u}{t}\br{t_0,x_0}\defeq\br{u\br{\cdot,x_0}}^{\Delta}\br{t_0}$. Analogously, the classical partial derivative $\frac{\partial{u}}{\partial{x}}$ of $u$ with respect to the second variable is defined.

Let $T>0$. Let $F\st\sbrt{0,T}\times\lsrbr{0,\infty}\to\R$, $\psi\st\sbrt{0,T}\to\R$ and $\varphi\st\lsrbr{0,\infty}\to\R$. Consider the following parabolic partial dynamic equation, together with a boundary and an initial condition, respectively
\begin{subequations}\label{eq_04_06}
	\begin{align}
		\pDelta{u}{t}u\br{t,x}&=\diffp[2]{u}{x}\br{t,x}+F\br{t,x},\quad{t}\in\sbrtk{0,T},x\in\lsrbr{0,\infty},\label{eq_04_06_01}\\
		u\br{t,0}&=\psi\br{t},\quad{t}\in\sbrt{0,T},\\
		u\br{0,x}&=\varphi\br{x},\quad{x}\in\lsrbr{0,\infty}.
	\end{align}
\end{subequations}

Now, we semi-discretise \eqref{eq_04_06} in the second variable. For simplicity, we consider unit step. Define $u_n\br{\cdot}\defeq{u}\br{\cdot,n}$, $F_n\br{\cdot}\defeq{F}\br{\cdot,n}$ and $\varphi_n\defeq\varphi\br{n}$, for $n=0,1,\dots$

For $t\in\sbrt{0,T}$ and $n=1,2,\dots$, consider the central difference approximation of the first term on the RHS of \eqref{eq_04_06_01} at the point $\br{t,n}$
\begin{equation}\label{eq_04_07}
	\diffp[2]{u}{x}\br{t,n}\approx{u}_{n+1}\br{t}-2u_n\br{t}+u_{n-1}\br{t}\eqdef\Lambda_n\br{t,\br{u_k\br{t}}_{k=1}^{\infty}}.
\end{equation}
Substituting \eqref{eq_04_07} into \eqref{eq_04_06_01}, we obtain the following approximation of \eqref{eq_04_06}
\begin{subequations}\label{eq_04_08}
	\begin{align}
		u_n^{\Delta}\br{t}&=\Lambda_n\br{t,\br{u_k\br{t}}_{k=1}^{\infty}}+F_n\br{t},\quad{t}\in\sbrtk{0,T},n=1,2,\dots,\\
		u_0\br{t}&=\psi\br{t},\quad{t}\in\sbrt{0,T},\label{eq_04_08_02}\\
		u_n\br{0}&=\varphi_n,\quad{n}=1,2,\dots
	\end{align}
\end{subequations}
Now, \eqref{eq_04_08} can be rewritten as \eqref{eq_04_03}. In particular, consider the following IVP
\begin{subequations}\label{eq_04_09}
	\begin{align}
		\bu^{\Delta}\br{t}&=\br{\Lambda_k\br{t,\bu\br{t}}+F_k\br{t}}_{k=1}^{\infty}\eqdef\bff\br{t,\bu\br{t}},\quad{t}\in\sbrtk{0,T},\label{eq_04_09_01}\\
		\bu\br{0}&=\br{\varphi_k}_{k=1}^{\infty}\eqdef\bvarphi.
	\end{align}
\end{subequations}

Note that the term $u_0\br{\cdot}$, which occurs in $\Lambda_1\br{\cdot,\br{u_k\br{\cdot}}_{k=1}^{\infty}}$ (see \eqref{eq_04_07}), is not part of $\br{u_k\br{\cdot}}_{k=1}^{\infty}$. Rather, the term $u_0\br{\cdot}$ is given by \eqref{eq_04_08_02}, i.e.
\begin{equation}\label{eq_04_10}
	\Lambda_1\br{\cdot,\bx}=x_2-2x_1+\psi\br{\cdot}.
\end{equation}
Also, note that in such a setting, $\Lambda_1$ is the only \enquote{$\Lambda_k$} depending explicitly on $t$ (since according to \eqref{eq_04_07}, $\Lambda_k\br{\cdot,\bx}=x_{k+1}-2x_k+x_{k-1}$, for all $k=2,3,\dots$).

\begin{example}\label{ex_05}
	Consider the IVP \eqref{eq_04_09}. Assume that
	\begin{enumerate}[(i)]
		\item\label{ex_05_assum_01} The mapping $g\st{t}\mapsto{F}\br{t,\cdot}$, for $t\in\sbrt{0,T}$ (i.e. $\br{g\br{t}}\br{x}\defeq{F}\br{t,x}$, for $t\in\sbrt{0,T}$, for $x\in\lsrbr{0,\infty}$) is such that $g\in\Crd{\sbrt{0,T}}{\BC{\lsrbr{0,\infty}}{\R}}$.
		\item\label{ex_05_assum_02} $\lim_{x\to\infty}F\br{t,x}=0$, for all $t\in\sbrt{0,T}$.
		\item\label{ex_05_assum_03} $\varphi\br{\cdot}$ is continuous on $\lsrbr{0,\infty}$, such that $\lim_{x\to\infty}\varphi\br{x}=0$.
		\item\label{ex_05_assum_04} $\psi\equiv0$ on $\sbrt{0,T}$.
	\end{enumerate}
	Utilising \theoremRef{thm_11}, we will show that under these assumptions, the IVP \eqref{eq_04_09} has a local solution in the space $c_0$. Also, note that \assumptionRef{ex_05_assum_04} is only for convenience and it could be easily extended to the case of $\psi\in\C{\sbrt{0,T}}{\R}$.
	
	Firstly, note that \assumptionRef{ex_05_assum_02} implies that $\bff\br{t,\bx}\in{c}_0$ (see also \eqref{eq_04_09_01} and \eqref{eq_04_07}), for all $\br{t,\bx}\in\sbrt{0,T}\times{c}_0$. Similarly, \assumptionRef{ex_05_assum_03} implies that $\bvarphi\in{c}_0$. Thus, the problem is well-formulated. Moreover, according to \eqref{eq_04_10}, \assumptionRef{ex_05_assum_04} implies that $\Lambda_1\br{\cdot,\bx}=x_2-2x_1$, for all $\bx\in{c}_0$.
	
	Let $s,t\in\sbrt{0,T}$ and $\bx\in{c}_0$ be arbitrary. Then the following estimate holds
	\begin{equation*}
		\begin{aligned}
			&\norm{\bff\br{s,\bx}-\bff\br{t,\bx}}_{c_0}\eqText{\eqref{eq_04_09_01}}\sup_{k=1,2,\dots}\abs{\Lambda_k\br{s,\bx}-\Lambda_k\br{t,\bx}+F_k\br{s}-F_k\br{t}}\\
			&\qquad\eqText{\eqref{eq_04_07}}\sup_{k=1,2,\dots}\abs{F_k\br{s}-F_k\br{t}}\leq\sup_{x\in\lsrbr{0,\infty}}\abs{F\br{s,x}-F\br{t,x}}\\
			&\qquad\eqText{\assumRef{ex_05_assum_01}}\norm{g\br{s}-g\br{t}}_{\BC{\lsrbr{0,\infty}}{\R}}.
		\end{aligned}
	\end{equation*}
	Taking into account \lemmaRef{lm_01} and \assumptionRef{ex_05_assum_01} of \exampleRef{ex_05}, it is easy to see that \assumptionRef{thm_11_assum_01} of \theoremRef{thm_11} is satisfied (see also \definitionRef{def_04}).
	
	Now, let $t\in\sbrt{0,T}$ and $\bx,\by\in{c}_0$ be arbitrary. Then the following estimate holds
	\begin{equation*}
		\norm{\bff\br{t,\bx}-\bff\br{t,\by}}_{c_0}\eqText{\eqref{eq_04_09_01}}\sup_{k=1,2,\dots}\abs{\Lambda_k\br{t,\bx}-\Lambda_k\br{t,\by}}\leqText{\eqref{eq_04_07}}4\cdot\sup_{k=1,2,\dots}\abs{x_k-y_k}=4\cdot\norm{\bx-\by}_{c_0},
	\end{equation*}
	which immediately implies that \assumptionRef{thm_11_assum_02} of \theoremRef{thm_11} is satisfied as well.
	
	Finally, define $p_j\br{\cdot}\defeq\abs{F_j\br{\cdot}}=\abs{F\br{\cdot,j}}$, for $j=1,2,\dots$ Using \assumptionRef{ex_05_assum_01} of \exampleRef{ex_05}, it is easy to verify that $p_j\in\Crd{\sbrt{0,T}}{\R}$, for all $j=1,2,\dots$ Furthermore, defining $n_1\defeq1$, $n_j\defeq{j}-1$, for $j=2,3,\dots$, and $Q\defeq4$, we immediately obtain that \assumptionRef{thm_11_assum_03} of \theoremRef{thm_11} is satisfied. Finally, \assumptionRef{thm_11_assum_04} of \theoremRef{thm_11} is an immediate consequence of \assumptionRef{ex_05_assum_02} of \exampleRef{ex_05}, whereas \assumptionRef{thm_11_assum_05} of \theoremRef{thm_11} follows from the fact that $g$ (as defined in \assumptionRef{ex_05_assum_01} of \exampleRef{ex_05}) is bounded on $\sbrt{0,T}$ (since it is rd-continuous on $\sbrt{0,T}$).
	
	Thus, according to \theoremRef{thm_11}, \eqref{eq_04_09} has a local solution in the space $c_0$.
\end{example}

\begin{remark}
	In \exampleRef{ex_05}, the Banach space $\E=c_0$ was chosen, as it has known closed form for the Hausdorff measure of noncompactness (see \eqref{eq_04_01}), which is suitable for applications. Other spaces that could be considered are e.g. $c$ or $\ell^p$, for which such closed forms for the Hausdorff measure of noncompactness are also known and are described in \cite{Banas_Mursaleen}. Also, the space $\ell^{\infty}$ might be considered, as it has a known form of a sublinear (although non-Hausdorff) measure of noncompactness as well (see \cite{Banas_Mursaleen}). However, in order to keep the example concise, $c_0$ was chosen among these spaces.
\end{remark}

\begin{remark}
	The space domain in \eqref{eq_04_06} can be also considered to be the whole real line, instead of $\lsrbr{0,\infty}$. In such a case, semi-discretisation of \eqref{eq_04_06} would lead to a dynamic equation in a Banach space of biinfinite sequences on $\Z$. However, this would require to develop the corresponding theory of measures of noncompactness for such spaces (although, we expect that the generalisation of the theory from sequence spaces to the case of biinfinite sequence spaces on $\Z$ should be straightforward). Thus, for the sake of simplicity, we work only with the \enquote{usual} sequence spaces (and thus we restrict the space variable in \eqref{eq_04_06} to the half-infinite interval $\lsrbr{0,\infty}$). Finally, note that countable systems of dynamic equations in such biinfinite sequence spaces (mainly $\ell^{\infty}\br{\Z}$) are studied in \cite{Slavik_Stehlik_2015, Slavik_Stehlik_2014, Slavik_Stehlik_Volek}.
\end{remark}

\section{Concluding remarks}\label{sec_05}
Our main result, \theoremRef{thm_08}, generalises Theorem $13.3.1$ in \cite{Banas_Goebel}. Indeed, if we consider $\T=\R$, $u_0=0_{\E}$ and $f$ to be uniformly continuous, we obtain the result from \cite{Banas_Goebel} as a direct consequence of \theoremRef{thm_08}. Moreover, compared to \cite{Banas_Goebel}, we require an additional assumption, namely \assumptionRef{thm_08_assum_05} in \theoremRef{thm_08}. This is due to the fact that in \cite{Banas_Goebel}, it is assumed that $u_0=0_{\E}$, and as one can easily verify directly from \definitionRef{def_06}, for every sublinear measure of noncompactness it holds that $\mu_{\E}\br{\cbr{0_{\E}}}=0$. However, note that \assumptionRef{thm_08_assum_05} in \theoremRef{thm_08} is not redundant, as there exists a sublinear measure of noncompactness that does not satisfy this assumption. As an example, consider the mapping $\mu_{\E}\st{X}\mapsto\sup_{x\in{X}}\norm{x}_{\E}$, for $X\in\fM_{\E}$. According to Section $3.1$ in \cite{Banas_Goebel}, such a $\mu_{\E}$ is a sublinear measure of noncompactness, and it is rather easy to see that $\mu_{\E}\br{X}=0$ if and only if $X=\cbr{0_{\E}}$.

As mentioned at the very beginning, the requirements on $f$ in \theoremRef{thm_08} are of a different type when compared to the other existence results found in the literature. Also, the use of a Kamke $\Delta$-function in the comparison condition of our main theorem is completely new, and was firstly used in this paper. The use of such a general comparison condition, together with an arbitrary (axiomatically defined) sublinear measure of noncompactness, puts \theoremRef{thm_08} into quite a general setting.

Even though we worked with the Cauchy $\Delta$-integral, the whole theory can be also developed in the framework of the more general Riemann $\Delta$-integral. Indeed, notice that throughout this paper, we use the Cauchy $\Delta$-integral only for rd-continuous functions, for which it is well-known that these two integrals coincide (however, as the simpler Cauchy $\Delta$-integral was sufficient, we decided to work with that one). Thus, the whole theory presented in this paper, including \theoremRef{thm_08}, can be stated in the context of Riemann $\Delta$-integral as well. Hence, \theoremRef{thm_08} indeed extends the result from \cite{Banas_Goebel} (which uses Riemann integration for the case of $\T=\R$).

For convenience, we consider the initial condition \eqref{eq_03_01_02} only at the left-most point $a\in\sbrt{a,b}$. The extension of \theoremRef{thm_08} to the case of an arbitrary starting point $t_0\in\sbrt{a,b}$ should be straightforward, yet one needs to treat such a case carefully. In particular, in order to guarantee the existence of a solution of \eqref{eq_03_01}, with the initial condition \eqref{eq_03_01_02} being at some $t_0\in\brt{a,b}$, on some non-trivial interval $\sbrt{\aast,t_0}$ to the left of $t_0$ (note that the existence of a solution on some non-trivial interval $\sbrt{t_0,\bast}$ to the right of $t_0$ would be guaranteed by \theoremRef{thm_08}), the assumptions of \theoremRef{thm_08} should be sufficient for the case of $t_0$ being left-dense. However, in the case of $t_0$ being left-scattered, some assumption of \textit{regressivity} of $f$ would be needed (see also Definition $8.15$ and Theorem $8.18$ in \cite{Bohner_Dynamic}). This is in contrast to the case of $t_0$ being right-scattered, where the existence of a solution to the right of $t_0$ (on some non-trivial interval) is guaranteed trivially.

\theoremRef{thm_01} is a well-known result which is proved e.g. in \cite{Cichon_Kubiaczyk_SikorskaNowak_Yantir_2012}. However, as we use it in a slightly different setting and provide a different proof, we decided to prove it here as well.

In \lemmaRef{lm_01} we provide an equivalent \enquote{$\eps\text{-}\delta$ definition} for rd-continuity, which we did not find in the existing literature. This equivalent definition is then utilised when defining the uniform rd-equicontinuity (see \definitionRef{def_04}), which was introduced by us as a new concept as well. Moreover, note that the uniform rd-equicontinuity generalises the notion of the uniform equicontinuity. In particular, the relationship between a family of functions being uniformly rd-equicontinuous and the individual functions being rd-continuous, is precisely the same as the relationship between a family of functions being uniformly equicontinuous and the individual functions being continuous. Finally, note that a condition similar to the one provided in \lemmaRef{lm_01} was presented for the so called \textit{regulated} functions at the beginning of the proof of Theorem $5.21$ in \cite{Bohner_Advances}.

In \remarkRef{rem_06} we discuss why the restriction of an rd-continuous function to some time scale subinterval is not necessarily rd-continuous. In \remarkRef{rem_09} we discuss what implications this might have on the relationship between local solutions of the IVP \eqref{eq_02_17} and solutions of the corresponding Volterra-type integral equation \eqref{eq_02_18}. Both these remarks are supported by counterexamples (see \exampleRef{ex_02} and \exampleRef{ex_03}, respectively). Although this is more a technical thing which can be easily overcome (see \remarkRef{rem_07}), we did not find such a discussion in the existing literature.

In \theoremRef{thm_03} we prove the equivalence between local solutions of the IVP \eqref{eq_02_17} and solutions of the corresponding Volterra-type integral equation \eqref{eq_02_18}. \theoremRef{thm_04} and \theoremRef{thm_05} are direct consequences of this result. Such an equivalence is well-known in the literature, although in a different setting (see e.g. \cite{Cichon_Kubiaczyk_SikorskaNowak_Yantir_2012, Kubiaczyk_SikorskaNowak}). Thus, we present and prove it ourselves in order to stress the technical discrepancies described in the previous paragraph, and to show how they can be overcome.

\lemmaRef{lm_03} is based on Lemma $13.2.1$ in \cite{Banas_Goebel}, where it is stated in a slightly different form. For completeness, we provide the proof ourselves. However, note that the proof we present mimics the steps of the proof of the above-mentioned result from \cite{Banas_Goebel}.

\lemmaRef{lm_04} provides a new result in the context of time scales and the theory of measures of noncompactness. The proof was inspired by the proof of Lemma $3$ in \cite{Oberta} and that of Lemma $13.2.2$ in \cite{Banas_Goebel}. Finally, note that a similar relation as in \lemmaRef{lm_04}, yet in a different setting (e.g. notice the use of the newly introduced notion of uniform rd-equicontinuity), in the context of time scales was also shown e.g. in \cite{Kubiaczyk_SikorskaNowak, SikorskaNowak}.

\theoremRef{thm_06} is an extension of Theorem $11.2$ in \cite{Banas_Goebel} to the case of an arbitrary time scale. However, note that in \cite{Banas_Goebel}, the proof of the first $5$ axioms of \definitionRef{def_05} is skipped. Thus, the proof presented by us supplements also the case of $\T=\R$ proved in \cite{Banas_Goebel}, as we provide the proofs for these axioms as well. Also, note that in the proof of \theoremRef{thm_06}, verification of \axiomRef{def_05_ax_06} of \definitionRef{def_05} follows precisely the above-mentioned proof from \cite{Banas_Goebel}.

The notion of a Kamke $\Delta$-function (see \definitionRef{def_07}) is a new concept in the theory of time scales. It extends the notion of a Kamke function presented in Definition $3.1$ in \cite{Li} and Section $13.3$ in \cite{Banas_Goebel}, to the case of an arbitrary time scale. In particular, note that for $\T=\R$, Kamke $\Delta$-function reduces to these two above mentioned notions of Kamke function. In \theoremRef{thm_10} we provide a particular example of a Kamke $\Delta$-function, which for $\T=\R$ reduces to the well-known example of a Kamke function usually found in the literature. However, it might be an interesting question to find more sophisticated examples of Kamke $\Delta$-functions, e.g. one which is not of the form $q\br{t}\cdot{h}\br{x}$, or one for which $u\equiv0$ is not the unique non-negative continuous solution of \eqref{eq_02_46} (and thus the additional requirement regarding $\Delta$-differentiability at $a$ from \axiomRef{def_07_ax_04} of \definitionRef{def_07} would need to be employed).

\theoremRef{thm_11} was inspired by Section $3$ in \cite{Banas_Lecko} and Theorem $17$ in \cite{Oberta}, where similar systems of classical differential equations and fractional differential equations, respectively, are studied. Also, note that \sectionRef{sec_04_03} was inspired by Section $4.3$ in \cite{Oberta}, where a similar fractional parabolic PDE (for $\T=\R$) is investigated. Moreover, note that a more general non-linearity, e.g. in the form of a $p$-Laplacian, could be considered in \exampleRef{ex_05} in the same way as it is done in \cite{Oberta}. Also, note that partial dynamic equations of the type \eqref{eq_04_06}, although with a more general non-linear dependence on $u$ in the form of $F\br{t,x,u\br{t,x}}$, could be studied in the future. Finally, note that the existence and uniqueness theorems for reaction-diffusion equations, together with their dependence on the \textit{graininess} function, involving such a general non-linear dependence on $u$ are studied e.g. in \cite{Slavik_Stehlik_Volek}.



\small{
\vspace{5mm}
\noindent\textbf{Acknowledgements}\hspace{3mm}The author would like to thank prof. Pavel Řehák for his support, guidance and for his insightful remarks and suggestions. Also, the author would like to thank the two anonymous reviewers for their helpful and constructive comments.

\vspace{2mm}
\noindent\textbf{Funding}\hspace{3mm}The research has been supported by the Brno University of Technology grant number FSI-S-23-8161.

\vspace{2mm}
\noindent\textbf{Conflict of interest}\hspace{3mm}The author declares that he has no conflict of interest.

\vspace{2mm}
\noindent\textbf{Data Availability Statement}\hspace{3mm}Data sharing not applicable to this article as no datasets were generated or analysed during the current study.
\vspace{3mm}}

\end{document}